\documentclass{amsart}

\usepackage[normalem]{ulem}
\usepackage{amsmath,amsthm, amssymb,mathrsfs,  xypic, url, verbatim,stmaryrd}
\usepackage{xr-hyper}
\usepackage{hyperref}
\usepackage{stackengine,scalerel}
\usepackage{todonotes}
\presetkeys{todonotes}{inline, size=\small}{}
\usepackage[capitalize]{cleveref}
\usepackage{graphicx}
\usepackage{quiver}
\usepackage{tikz}
\usepackage{enumitem}
\usetikzlibrary{matrix, arrows}
\usetikzlibrary{cd}

\crefname{maintheorem}{Theorem}{Theorems}
\crefname{maincorollary}{Corollary}{Corollaries}
\crefname{conjecture}{Conjecture}{Conjectures}

\newcommand{\mbb}[1]{\mathbb{#1}}

\makeatletter
\newcommand*{\addFileDependency}[1]{
  \typeout{(#1)}
  \@addtofilelist{#1}
  \IfFileExists{#1}{}{\typeout{No file #1.}}
}
\makeatother

\addFileDependency{main-1.tex}
\addFileDependency{main-1.aux}

\begin{document}
\newcommand{\actsonr}{\mathrel{\reflectbox{$\righttoleftarrow$}}}
\newcommand{\actsonl}{\mathrel{\reflectbox{$\lefttorightarrow$}}}

\newcommand{\floor}[1]{\lfloor #1 \rfloor}

\newcommand{\isoeq}{\cong}
\newcommand{\cech}{\vee}
\newcommand{\dsum}{\mathop{\oplus}}
\newcommand{\End}{\mathrm{End}}

\newcommand{\calQ}{\mathcal{Q}}
\newcommand{\calO}{\mathcal{O}}
\newcommand{\calM}{\mathcal{M}}

\newcommand{\frakm}{\mathfrak{m}}

\newcommand{\bbA}{\mathbb{A}}
\newcommand{\bbB}{\mathbb{B}}
\newcommand{\bbC}{\mathbb{C}}
\newcommand{\bbD}{\mathbb{D}}
\newcommand{\bbE}{\mathbb{E}}
\newcommand{\bbF}{\mathbb{F}}
\newcommand{\bbG}{\mathbb{G}}
\newcommand{\bbH}{\mathbb{H}}
\newcommand{\bbI}{\mathbb{I}}
\newcommand{\bbJ}{\mathbb{J}}
\newcommand{\bbK}{\mathbb{K}}
\newcommand{\bbL}{\mathbb{L}}
\newcommand{\bbM}{\mathbb{M}}
\newcommand{\bbN}{\mathbb{N}}
\newcommand{\bbO}{\mathbb{O}}
\newcommand{\bbP}{\mathbb{P}}
\newcommand{\bbQ}{\mathbb{Q}}
\newcommand{\bbR}{\mathbb{R}}
\newcommand{\bbS}{\mathbb{S}}
\newcommand{\bbT}{\mathbb{T}}
\newcommand{\bbU}{\mathbb{U}}
\newcommand{\bbV}{\mathbb{V}}
\newcommand{\bbW}{\mathbb{W}}
\newcommand{\bbX}{\mathbb{X}}
\newcommand{\bbY}{\mathbb{Y}}
\newcommand{\bbZ}{\mathbb{Z}}

\newcommand{\bc}{\mathbf{c}}

\newcommand{\Spec}{\mathrm{Spec}}\
\newcommand{\Triv}{\mathrm{Triv}}
\newcommand{\Loc}{\mathrm{Loc}}
\newcommand{\Et}{\mathrm{Et}}

\newcommand{\ul}[1]{\underline{#1}}

\newcommand{\Vect}{\mathrm{Vect}}
\newcommand{\FilVect}{\mathrm{FilVect}}
\newcommand{\FilVB}{\mathrm{FilVB}}

\newcommand{\trFil}{\mathrm{trFil}}
\newcommand{\gr}{\mathrm{gr}}
\newcommand{\LattFilt}[1]{\mr{Latt}_{#1}\mr{Vect}}

\newcommand{\dR}{\mathrm{dR}}

\newcommand{\HS}{\mathrm{HS}}

\newcommand{\Hodge}{\mathrm{Hdg}}

\newcommand{\BC}{\mathrm{BC}}
\newcommand{\CM}{\mathrm{CM}}

\newcommand{\goodred}{\textrm{good-red}}

\newcommand{\dash}{\mathrm{-}}
\newcommand\+{\mkern3mu}

\newcommand{\rev}[1]{{\color{teal} #1}}
\newcommand{\strike}[1]{{\color{gray} \sout{#1}}}
\newcommand{\spc}[1]{{\color{blue} \textsf{$\blacktriangle\blacktriangle\blacktriangle$ Comment: [#1]}}}

\newcommand{\Iref}[1]{I.\cref{I.#1}}

\newcommand{\triv}{\mr{triv}}
\newcommand{\crys}{\mathrm{crys}}
\newcommand{\Fil}{\mathrm{Fil}}
\newcommand{\calL}{\mathcal{L}}
\newcommand{\Stab}{\mathrm{Stab}}
\newcommand{\Spd}{\mathrm{Spd}}

\newcommand{\HT}{\mathrm{HT}}
\newcommand{\GH}{\mathrm{GH}}
\newcommand{\GM}{\mathrm{GM}}
\newcommand{\Gr}{\mathrm{Gr}}
\newcommand{\Fl}{\mathrm{Fl}}
\newcommand{\LT}{\mathrm{LT}}
\newcommand{\colim}{\mathrm{colim}}
\newcommand{\cris}{\mathrm{cris}}
\renewcommand{\l}{\left}
\renewcommand{\r}{\right}
\newcommand{\GL}{\mathrm{GL}}
\newcommand{\qpet}{\mr{qpet}}

\newcommand{\Perf}{\mathrm{Perf}}
\newcommand{\Perfd}{\mathrm{Perfd}}
\newcommand{\mc}[1]{\mathcal{#1}}

\newcommand{\mr}[1]{\mathrm{#1}}
\newcommand{\mf}[1]{\mathfrak{#1}}
\newcommand{\ms}[1]{\mathscr{#1}}

\newcommand{\ol}[1]{\overline{#1}}

\newcommand{\Kt}{\mathrm{Kt}}
\newcommand{\Isoc}{\mr{Isoc}} 

\newcommand{\Spa}{\mathrm{Spa}}
\newcommand{\Spf}{\mathrm{Spf}}
\newcommand{\perf}{\mathrm{perf}}
\newcommand{\Spr}{\mr{Spr}}
\newcommand{\Hom}{\mr{Hom}}
\newcommand{\Ker}{\mr{Ker}}
\newcommand{\adm}{\mr{adm}}
\newcommand{\bs}{\backslash}
\newcommand{\BB}{\mathrm{BB}}
\newcommand{\an}{\mathrm{an}}
\newcommand{\AdmFil}{\mathrm{AdmFil}} 

\newcommand{\lfid}{{\diamond_{\mathrm{lf}}}}
\newcommand{\lf}{\mathrm{lf}}

\newcommand{\QQ}{\mathbb{Q}}
\newcommand{\ZZ}{\mathbb{Z}}

\newcommand{\FF}{\mathrm{FF}}

\newcommand{\proet}{\mr{pro\acute{e}t}}
\newcommand{\et}{\mr{\acute{e}t}}
\newcommand{\fet}{\mr{f\acute{e}t}}

\newcommand{\badm}{{b\textrm{-adm}}}
\newcommand{\Qp}{{\mbb{Q}_p}}
\newcommand{\Qpbreve}{{\breve{\mbb{Q}}_p}}
\newcommand{\Qpbrevebar}{{\overline{\breve{\mbb{Q}}}_p}}

\newcommand{\Hdg}{\mr{Hdg}}
\newcommand{\Hdggen}{\mr{gen}}
\newcommand{\MT}{\mr{MT}}
\newcommand{\Rep}{\mathrm{Rep}}

\newcommand{\FilPhiMod}{\mathrm{MF}^{\varphi}}
\newcommand{\waFilPhiMod}{\mathrm{MF}^{\varphi, \mr{wa}}}

\newcommand{\Isom}{\mr{Isom}}
\newcommand{\Gal}{\mr{Gal}}
\newcommand{\Aut}{\mr{Aut}}
\newcommand{\AdmPair}{\mr{AdmPair}}
\newcommand{\Forget}{\mr{Forget}}
\newcommand{\HN}{\mr{HN}}
\newcommand{\HNsplit}{\mr{HN}\dash\mr{split}}

\newcommand{\Cont}{\mr{Cont}}
\newcommand{\AP}{\mathbf{A}}
\newcommand{\GAP}{\mathbf{A}}
\newcommand{\ev}{\mathrm{ev}}
\newcommand{\std}{\mathrm{std}}
\newcommand{\Lie}{\mr{Lie}}
\newcommand{\Sh}{\mr{Sh}}
\newcommand{\univ}{\mr{univ}}
\newcommand{\MF}{\mathrm{MF}}
\newcommand{\can}{\mr{can}}
\newcommand{\MG}{\mr{MG}}
\newcommand{\VHS}{\mr{VHS}}

\newcommand{\spck}[1]{{\color{teal} \textsf{$\dagger\dagger\dagger$ CK: [#1]}}}
\newcommand{\bilatticed}{\mathrm{bl}}
\newcommand{\latticed}{\mathrm{latticed}}
\newcommand{\filt}{\mathrm{filt}}
\newcommand{\FamHS}{\mathrm{FamHS}}
\newcommand{\basic}{\mr{basic}}
\newcommand{\lsi}{\mathbb{L}}
\newcommand{\M}{\mathcal{M}}

\newcommand{\ab}{\mathrm{ab}}\newcommand{\der}{\mathrm{der}}

\newcommand{\impi}{\mathrm{D}}

\numberwithin{equation}{subsubsection}
\theoremstyle{plain}

\newtheorem{maintheorem}{Theorem} 
\renewcommand{\themaintheorem}{\Alph{maintheorem}} 
\newtheorem{maincorollary}[maintheorem]{Corollary}
\newtheorem{mainconjecture}[maintheorem]{Conjecture}

\newtheorem*{theorem*}{Theorem}

\newtheorem{theorem}[subsubsection]{Theorem}
\newtheorem{corollary}[subsubsection]{Corollary}
\newtheorem{conjecture}[subsubsection]{Conjecture}
\newtheorem{proposition}[subsubsection]{Proposition}
\newtheorem{lemma}[subsubsection]{Lemma}

\theoremstyle{definition}

\newtheorem{example}[subsubsection]{Example}
\newtheorem{assumption}[subsubsection]{Assumption}
\newtheorem{definition}[subsubsection]{Definition}
\newtheorem{remark}[subsubsection]{Remark}
\newtheorem{question}[subsubsection]{Question}

\renewcommand{\Vect}{\mathrm{Vect}}

\title[Admissible pairs and $p$-adic Hodge structures III]{Admissible pairs and $p$-adic Hodge structures III: Variation and unlikely intersection}

\author{Sean Howe}
\email{sean.howe@utah.edu}
\author{Christian Klevdal}
\email{klevdal@math.utah.edu}

\begin{abstract} 
We extend the relative theory of admissible pairs and $p$-adic Hodge structures introduced in Part II to allow variation in the underlying local systems of $\mathbb{Q}_p$-vector spaces and isocrystals. This extension accommodates, in particular, the families of $p$-adic Hodge structures that arise from the cohomology of certain smooth proper families. Such a variation gives rise to a cover with a global Hodge period map, and we study this cover and its period map from a differential perspective both classically and via the theory of inscription. This study is motivated by our transcendence results in Parts I and II and analogies with complex bi-algebraic geometry, and we also extend these ideas in other directions: First, we study the locus of special points on local Shimura varieties. We establish a refined version of a prediction of Rapoport-Viehmann on the density of special points, but give a robust counter-example to the local analog of the Andr\'{e}-Oort conjecture in this setting. This shuts down a broader theory of unlikely intersection but leaves open the possibility of stronger geometric transcendence results than the bi-analytic Ax--Lindemann theorem of Part II. Using the Banach-Colmez Tangent Bundles for infinite level local Shimura varieties that arise from the theory of inscription, we define precise notions of generic and exceptional intersections and then formulate an Ax--Schanuel conjecture that we expect to refine our Ax--Lindemann theorem.   
\end{abstract}
\maketitle

\tableofcontents

\section{Introduction}

In Part I \cite{howe-klevdal:ap-pahsI} of this series of papers, we defined a category of \emph{$p$-adic Hodge structures} that provides a natural structural analog of the category of real Hodge structures, and then developed a local $p$-adic transcendence theory by comparing $p$-adic Hodge structures with a toy category of cohomological motives, the \emph{admissible pairs}. The main result of Part I \cite[Theorem B]{howe-klevdal:ap-pahsI} is a characterization of complex multiplication for basic admissible pairs in terms of the transcendence properties of their periods. This is a local, $p$-adic analog of results of Schneider \cite{schneider:j}, Cohen \cite{cohen:transcendence}, and Shiga--Wolfart \cite{shiga-wolfart:transcendence} for abelian varieties, and generalizes an earlier result of one of the authors \cite{howe:transcendence} in the $p$-adic setting for one-dimensional formal groups. 

The theory of Part I depends on the choice of an algebraically closed perfectoid field $C/\mathbb{Q}_p$. In Part II, we gave a na\"{i}ve relativization of this theory by introducing a category of \emph{neutral} $p$-adic Hodge structures and admissible pairs over an arbitrary locally spatial diamond, then used this to study infinite level local Shimura varieties and their non-minuscule analogs from a Tannakian perspective as moduli spaces of admissible pairs and $p$-adic Hodge structures. The main result of Part II, \cite[Theorem A]{howe-klevdal:ap-pahsII}\footnote{See also \cite[Theorem 9.4]{howe-klevdal:ap-pahsII} for a generalization to the non-minuscule case.}, is a bi-analytic Ax--Lindemann theorem for infinite level local Shimura varieties. It implies that, in the basic case, the \emph{special subvarieties} of these moduli spaces, i.e.\ the loci where the objects parameterized have extra symmetries (as measured by a reduction of Tannakian structure group or by the existence of Hodge and Hodge-Tate tensors), are exactly the bi-analytic subvarieties, generalizing the case of special points obtained from the main result of Part I \cite[Corollary C]{howe-klevdal:ap-pahsI}.

Over a geometric point, a $p$-adic Hodge structure is a graded $B^+_\dR$-latticed $\mathbb{Q}_p$-vector space satisfying an extra property (analogous to the transversality of the Hodge filtration with its conjugate in complex Hodge theory), and an admissible pair is a $B^+_\dR$-latticed isocrystal satisfying an extra property. The relativization used in \cite{howe-klevdal:ap-pahsII} was na\"{i}ve in the sense that we did not allow any variation in the underlying $\mathbb{Q}_p$-vector space or underlying isocrystal. In particular, the notion of a neutral $p$-adic Hodge structure or admissible pair is poorly behaved from the perspective of descent. This notion was sufficient for the study of the moduli spaces treated in \cite{howe-klevdal:ap-pahsII}, however, any deeper study of these objects should allow for such a variation: for example, already in \cite[\S 10.1]{howe-klevdal:ap-pahsII}, in order to explain some possible generalizations of the bi-analytic Ax--Lindemann theorem, we introduced the notion of a \emph{variation of $p$-adic Hodge structure} over a smooth rigid analytic variety. This notion captures the families of $p$-adic Hodge structures that arise naturally from the cohomology of smooth proper families of rigid analytic varieties, but requires a non-trivial $\mathbb{Q}_p$-local system in place of the $\mathbb{Q}_p$-vector space that is used in the neutral theory.

The first purpose of this paper is to set up the general relative theory over a small $v$-sheaf, allowing variation in the underlying $\mathbb{Q}_p$-local system or isocrystal. The resulting categories of admissible pairs and $p$-adic Hodge structures are $v$-stacks, and \cref{theorem.ap-hs-properties} gives an equivalence between the $v$-stack of \emph{basic} admissible pairs and the $v$-stack of $p$-adic Hodge structures extending the equivalence over geometric points of \cite[Theorem 5.1.6]{howe-klevdal:ap-pahsI}. The neutral objects of \cite{howe-klevdal:ap-pahsII} form a subcategory, and, when the base is connected, it is a full subcategory; this subcategory was sufficient for the study of the moduli problems in \cite{howe-klevdal:ap-pahsII} because those moduli problems include trivializations of the underlying local system of $\mathbb{Q}_p$-vector spaces or isocrystals. When the base is $\Spd\, L$ for $L$ a $p$-adic field (that is, a complete discretely valued extension of $\mathbb{Q}_p$ with perfect residue field), Fontaine's category of admissible filtered $\varphi$-modules is equivalent to a full subcategory of unramified admissible pairs over $\Spd\, L$, and the comparison with crystalline Galois representations can be reframed in this language as a comparison with a full subcategory of variations of $p$-adic Hodge structure over $\Spd\, L$ (this is a rephrasing in our language of an observation of Fargues--Fontaine \cite{fargues-fontaine:courbes}, and was discussed for strict $p$-adic fields in \cite[\S5.5]{howe-klevdal:ap-pahsI}). 

The second purpose is to continue the exploration of variations of $p$-adic Hodge structure and more general transcendence theorems in $p$-adic bi-analytic geometry, along with the analogy between these and results in complex bi-algebraic theory, as begun already in \cite[\S 10.1]{howe-klevdal:ap-pahsII}. Our study here moves in three distinct directions.

\subsection{Variations of $p$-adic Hodge structure and period maps}
As described in \cite[\S 10.1]{howe-klevdal:ap-pahsII}, any variation of $p$-adic Hodge structure over a smooth rigid analytic variety gives rise to a covering space equipped with an \'{e}tale lattice period map lifting the Hodge period map along the Bialynicki-Birula morphism. We make a precise conjecture (\cref{conj.pot-good-red}) to the effect that this covering space is potentially unramified, and explain why we expect that a proof of this conjecture would lead to a generalization of the bi-analytic Ax--Lindemann theorem of \cite{howe-klevdal:ap-pahsII}. In particular, one expects that the derivative of the Hodge period map is simply the Kodaira-Spencer map and that maps from such a covering space to the $\mathbb{B}^+_\dR$-affine Grassmannian can be characterized as in the rigid analytic case \cite[Theorem 5.4-(4)]{howe-klevdal:ap-pahsII} (\cref{conj.pot-unram-prop}). Although we do not prove this conjecture, we can sidestep it to provide some evidence that it is reasonable: this covering space, which is a priori only a diamond, can be equipped with a natural differentiable structure in the form of an inscription as in \cite{Howe.InscriptionTwistorsAndPAdicPeriods}. The Hodge and \'{e}tale lattice period maps upgrade to maps of inscribed $v$-sheaves, and we explain how the Kodaira--Spencer map appears as the derivative of the inscribed Hodge period map while its canonical lift appears as the derivative of the inscribed \'{e}tale lattice period map. 

\subsection{Density of special points and Andr\'{e}-Oort}
The hyperbolic Ax--Lindemann theorem in complex bi-algebraic geometry plays a key role in the proof of the Andr\'{e}-Oort conjecture for global Shimura varieties. Given the bi-analytic Ax--Lindemann theorem \cite[Theorem A]{howe-klevdal:ap-pahsII}, it is natural to ask if there is a reasonable local Andr\'{e}-Oort conjecture for local Shimura varieties. In \cref{s.special-points}, we make a study of the structure of the special points that rules out the most na\"{i}ve such generalization. 

We first give a positive result:
\begin{theorem*}[\cref{theorem.special-points-dense} and \cref{corollary.density-of-special-points}] Let $G/\mathbb{Q}_p$ be a connected reductive group, let $[\mu]$ be a conjugacy class of minuscule cocharacters of $G_{\overline{\mathbb{Q}}_p}$, and let $b$ represent a class in the Kottwitz set $B(G,[\mu^{-1}])$. 
Let $\Fl_{[\mu^{-1}]}^{b-\adm}$ be the associated $b$-admissible locus and, for $K \leq G(\mathbb{Q}_p)$ a compact open subgroup, let $\mathcal{M}_{b,[\mu], K}$ be the associated local Shimura variety, an \'{e}tale covering space of $\Fl_{[\mu^{-1}]}^{b-\adm}$. The following are equivalent:
\begin{enumerate}
    \item Special points are rigid analytically Zariski dense in $\Fl_{[\mu^{-1}]}^{b-\adm}$.
    \item Special points are rigid analytically Zariski dense in $\mathcal{M}_{b, [\mu], K}$.
    \item $b$ represents the unique basic class in $B(G,[\mu^{-1}])$. 
\end{enumerate}
\end{theorem*}

The main point is to obtain the equivalence between (1) and (3). 
To establish it, we first show in \cref{prop.special-points-loc} that the special points in the admissible locus are contained in a finite set of $G_b(\mbb{Q}_p)$-orbits. In the basic case, any such orbit is dense, so the result follows once one has the existence of a single special point, which is handled in \cref{lemma.existence-special-point}. On the other hand, in the non-basic case, $G_b(\mbb{Q}_p)$ is contained in a proper Levi subgroup $M$ of $G_{\Qpbreve}$ and each of the finitely many orbits that appears is contained within a strict special subvariety associated to a $\mathbb{Q}_p$-form of $M$.  

\begin{remark}The density of special points was first anticipated in \cite[Remark 4.22 and Remark 5.5-(ii)]{rapoport-viehmann}, but without the basic hypothesis, however, it is easy to see that density fails already in the first non-basic example of the Serre-Tate disk (\cref{example.serre-tate}). Via the structural analogy developed in \cite{howe-klevdal:ap-pahsI}, it is not surprising that only the basic case behaves similarly to the classical theory of complex Shimura varieties: the basic case is the part of the theory related to $p$-adic Hodge structures and thus to the analogy with classical complex Hodge theory.
\end{remark}

Even in the non-basic case, these $M$-orbits are basic special subvarieties in which the special points are dense (\cref{remark.special-subvarieties}), so that the Zariski closure of all special points is a union of basic special subvarieties. Thus, a natural analog of the Andr\'{e}--Oort conjecture would be to guess that any irreducible component of the rigid Zariski closure of a collection of special points is a basic special subvariety. We exhibit a counter-example to this guess (\cref{example.andre-oort-counterexample}) already in the first non-trivial case by constructing a curve in the two dimensional height three Lubin--Tate space that contains a Zariski dense sets of special points but is not a special subvariety. The key point is that the special points in Lubin--Tate space can be characterized by a rationality property on their Hodge/Grothendieck-Messing/Gross-Hopkins period. This counterexample is robust and can be generalized to many other cases. 

\subsection{An Ax--Schanuel conjecture} In \cref{s.ax-schanuel}, we formulate a bi-analytic Ax--Schanuel type conjecture for infinite level local Shimura varieties (\cref{conj.ax-schanuel}). As explained in \cref{remark.ax-lindemann-relation}, we expect this conjecture to partially refine the bi-analytic Ax--Lindemann theorem \cite[Theorem A]{howe-klevdal:ap-pahsII}.

In complex bi-algebraic geometry as it arises in the study of complex variations of Hodge structure, an Ax--Schanuel theorem takes the following form: one embeds a complex analytic manifold $M$ into the analytification of an algebraic variety $V$, and then the Ax--Schanuel theorem uses special subvarieties of $M$ to explain exceptional intersections (i.e. intersections of higher than expected dimension) of $M$ with algebraic subvarieties of $V$. When $V=V_1 \times V_2$ and one of the maps $M \rightarrow V_i$ is a local isomorphism, such a result refines an Ax--Lindemann type theorem for the resulting bi-algebraic structure on $M$ in the sense of \cite{klingler-ullmo-yafaev:bialgebraic}. 

In order to make such a conjecture in our setting, we take our analog of the embedding of $M$ into $V$ to be the product of the Hodge and Hodge-Tate period maps on a basic infinite level local Shimura variety. Any infinite level local Shimura variety is expected to be a (pre)perfectoid space and this is known in many cases; in particular these spaces are very far from being rigid analytic. The main obstacle is thus to understand what a ``generic" intersection of such a space with a rigid analytic subvariety of the product of period domains should look like, so that one can give a precise definition of an exceptional (i.e. non-generic) intersection. To accomplish this, we use the Banach--Colmez Tangent Bundles computed via the theory of inscription introduced in \cite{Howe.InscriptionTwistorsAndPAdicPeriods} to make sense of a generically transverse intersection. We phrase the resulting conjecture (\cref{conj.ax-schanuel}) and related computations and discussion in language that is independent of the inscribed formalism in order to keep it more accessible. 

\subsection{Outline} 
In \cref{s.prelims} we recall some preliminaries that will be used throughout the paper. In \cref{s.admissible-pairs} we define our general categories of admissible pairs and $p$-adic Hodge structures and establish their basic properties. In particular, we prove the equivalence between basic admissible pairs and $p$-adic Hodge structures, compare with the neutral objects of \cite{howe-klevdal:ap-pahsII}, and define various period maps associated to a general $G$-admissible pair. In \cref{s.variations-padic-hs} we define admissible pairs of potential good reduction and variations of $p$-adic Hodge structure then discuss the relation between them; the main goal of the section is to explain and partially justify some geometric conjectures about these objects and to connect them to potential generalizations of the geometric transcendence results of \cite{howe-klevdal:ap-pahsII}. In \cref{s.special-points}, we prove our results on the density of special points and give a counter-example to the na\"{i}ve generalization of the Andr\'{e}-Oort conjecture to local Shimura varieties. Finally, in \cref{s.ax-schanuel} we formulate our Ax--Schanuel conjecture.

\subsection*{Acknowledgements} During the preparation of this work, Sean Howe was supported by the NSF through grant DMS-2201112 and DMS-2501816.  We thank Banff International Research Station
for hosting both authors for a Research in Teams project that led, in particular, to the formulation of \cref{conj.ax-schanuel}.

\section{Preliminaries}\label{s.prelims}

In this section we fix some notation and recall some preliminaries on $p$-adic geometry, local systems, and bundles on the Fargues--Fontaine curve. 

\subsection{Notation}

\subsubsection{}
Throughout this paper, $p$ is a fixed prime. We say a field $L$ is a non-archimedean field if it is complete with respect to a non-trivial non-archimedean absolute value and its residue field has characteristic $p$. We write $\mc{O}_L$ for the ring of integers of any non-archimedean field $L$, i.e. the elements of absolute value at most $1$. A $p$-adic field is a non-archimedean field containing $\mathbb{Q}_p$ that is complete with respect to a discrete valuation extending the $p$-adic absolute value on $\mathbb{Q}_p$ and has perfect residue field. We fix an algebraic closure $\ol{\mbb{F}}_p/\mathbb{F}_p$ and let $\breve{\QQ}_p = W(\ol{\mbb{F}}_p)[p^{-1}]$. We write $\sigma$ for the automorphism of $\breve{\QQ}_p$ induced by the $p$-power Frobenius on $\ol{\mbb{F}}_p$. If $A$ is a sheafy Huber ring, we write $\Spa\+ A$ for the adic space $\Spa(A, A^\circ)$. 

\subsubsection{}
If $G$ is an affine group scheme over a field $E$, we write $\Rep\+ G$ for the Tannakian category of algebraic representations of $G$ on finite dimensional $E$-vector spaces, and $\omega_\std \colon \Rep\+ G \to \Vect_E$ for the forgetful functor to $E$-vector spaces.

\subsubsection{Isocrystals}
We recall some results from \cite{kottwitz:isocrystals}. Let $\Isoc$ be the category of isocrystals over $\ol{\mbb{F}}_p$ (denoted $\Kt_{\QQ_p}$ in \cite{howe-klevdal:ap-pahsI}). The objects are given by pairs $(W, \varphi_W)$ consisting of a finite dimensional $\breve{\QQ}_p$-vector space and a $\sigma$-semilinear isomorphism $\varphi_W \colon W \xrightarrow{\sim} W$. The category $\Isoc$ is a semisimple (non-neutral) Tannakian category over $\QQ_p$, and the set of isomorphism classes of objects of $\Isoc$ is in bijection with $\mathbb{Q}$ by sending $\lambda = \frac{a}{b}$ with $b>0$ and $\gcd(a,b)=1$ to $D_\lambda = \breve{\QQ}_p[\pi]/(\pi^b - p^a)$ with $\varphi_{D_\lambda} = \pi \sigma$. 

If $G$ is a linear algebraic group over $\QQ_p$, a $G$-isocrystal is an exact tensor functor $\Rep\+ G \to \Isoc$. For $b \in G(\breve{\QQ}_p)$, we let $W_b \colon \Rep\+ G \to \Isoc$ be the $G$-isocrystal defined by $(V, \rho)\mapsto (V_{\breve{\QQ}_p}, \rho(b) \sigma)$. When $G$ is connected, any $G$-isocrystal is isomorphic to $W_b$ for some $b \in G(\breve{\QQ}_p)$, and the set of isomorphism classes of $G$-isocrystals is given by the Kottwitz set $B(G)$ of $\sigma$-conjugacy classes of $G(\breve{\QQ}_p)$. 

\subsection{$p$-adic geometry}
Let $\Perfd_{\QQ_p}$ be the category of perfectoid spaces over\footnote{This is equivalent to the category of perfectoid spaces in characteristic $p$ equipped with a structure morphism to $\Spd\+ \QQ_p$ \cite[Theorem 9.4.4]{scholze:berkeley}, so we could equally well work with perfectoid spaces in characteristic $p$ throughout this paper.} $\QQ_p$, which we view as a site equipped with the $v$-topology \cite[Definition 8.1]{scholze:etale-cohomology-of-diamonds}. A $v$-sheaf is a sheaf on the site $\Perfd_{\QQ_p}$. Every perfectoid space is a $v$-sheaf \cite[Theorem 8.7]{scholze:etale-cohomology-of-diamonds}, and a $v$-sheaf is said to be small if it admits a surjection from a perfectoid space \cite[Definition 12.1]{scholze:etale-cohomology-of-diamonds}. Any small $v$-sheaf $S$ admits an underlying topological space $|S|$ \cite[Definition 12.8]{scholze:etale-cohomology-of-diamonds} and there is a notion of (locally) spatial $v$-sheaves \cite[Definition 12.12]{scholze:etale-cohomology-of-diamonds}. If a $v$-sheaf $S$ is (locally) spatial, then $|S|$ is a (locally) spectral space \cite[Proposition 12.13]{scholze:etale-cohomology-of-diamonds}. A diamond\footnote{This is not the original definition of a diamond (c.f.\ \cite[Definition 11.1]{scholze:etale-cohomology-of-diamonds}) but it is equivalent \cite[Proposition 11.5]{scholze:etale-cohomology-of-diamonds}. We prefer to use this definition to avoid talking about the pro-\'etale topology on $\Perfd_{\QQ_p}$ and $v$-sheaves, since we will use the distinct notion the pro-\'etale site of a rigid analytic variety introduced in \cite{scholze:p-adic-ht}.} $S$ is a $v$-sheaf for which there exists a surjective quasi-pro-\'etale morphism from a perfectoid space to $S$. For any small $v$-sheaf $S$, we have the sites $S_v, S_\an$ consisting respectively of all $v$-sheaves over $S$ and all open $v$-subsheaves of $S$.

When $X$ is an adic space over $\Spa(\QQ_p)$, we write $X^\diamond$ for the functor on $\Perfd_{\mathbb{Q}_p}$, $X^\diamond(S) = \Hom(S, X)$ (morphisms in the category of adic spaces over $\QQ_p$), and we write $\Spd(A, A^+) := \Spa(A, A^+)^\diamond$. Because $X$ is an analytic adic space over $\QQ_p$, $X^\diamond$ is a locally spatial diamond \cite[Lemma 15.6]{scholze:etale-cohomology-of-diamonds}.

\subsubsection{Rigid analytic varieties}
If $L/\QQ_p$ is a non-archimedean field, a rigid analytic variety over $L$ is an adic space over $\Spa(L)$ that is locally topologically finite type, i.e.\ has an open cover by affinoids of the form $\Spa(A, A^\circ)$ for $A$ a topologically finite type $L$-algebra.

For a rigid analytic variety $X$ over $L$, we let $X_\et, X_\an$ be the \'etale site and analytic site of $X$. We write $X_{\proet}$ for the pro-\'etale site of \cite{scholze:p-adic-ht}, this is the full subcategory of pro-objects of $X_\et$ where the objects $\tilde{X} = \varprojlim_i X_i$ are formal inverse limits of small cofiltered systems $(X_i)_{i \in I}$ of representable objects $X_i \in X_\et$ such that the transition maps $X_i \to X_j$ are finite \'etale for $i > j$ large. The pro-system $(X_i)_{i \in I}$ is called a pro-\'etale presentation. Given an object $\tilde{X} = (X_i)_{i \in I} \in X_\proet$, we let $|\tilde{X}| = \varprojlim_{i \in I} |X_i|$ and $\tilde{X}^\diamond = \varprojlim_{i \in I} X_i^\diamond$. 

The functor $X \mapsto X^\diamond$ from rigid analytic varieties over $L$ to diamonds over $\Spd\+ L$ is fully faithful once restricted to the full subcategory of seminormal rigid analytic varieties \cite[Proposition 10.2.3]{scholze:berkeley}. 

\subsection{Local systems}
If $\Lambda$ is a topological ring, we let $\ul{\Lambda}$ be the $v$-sheaf on $\Perfd_{\QQ_p}$ whose value on a perfectoid space $U$ is $\ul{\Lambda}(U) = C^0(|U|, \Lambda)$. We let $\mc{O}$ be the structure sheaf whose value on a perfectoid space $U/\Spa(\mathbb{Q}_p)$ is $ \mc{O}_U(U)$, which is a $v$-sheaf by \cite[Theorem 8.7]{scholze:etale-cohomology-of-diamonds}. The de Rham period sheaf $\mbb{B}_\dR^+$ is the unique sheaf on $\Perfd_{\QQ_p}$ whose value on an affinoid perfectoid space $\Spa(A, A^+)$ is the completion of $W(A^{+\flat})[p^{-1}]$ along the kernel of the surjection $W(A^{+\flat})[p^{-1}] \to A$. (That this is a $v$-sheaf follows from \cite[Example 15.2.9-(5)]{scholze:berkeley}.) By construction, there is a canonical surjection $\theta \colon \mbb{B}_\dR^+ \to \mc{O}$, and the kernel is locally principal. We write $\mbb{B}_\dR = \mbb{B}_\dR^+[\ker(\theta)^{-1}]$. If $L$ is a $p$-adic field, there is a canonical inclusion $\ul{L} \hookrightarrow \mbb{B}^+_{\dR, \Spd\+ L}$. 

For $R \in \{\ul{\Lambda}, \mc{O}, \mbb{B}_\dR^{(+)} \}$, and $S$ a small $v$-sheaf, we write $R_S$ for the restriction of $R$ to $\Perfd_{S}$ (or, by a slight abuse of notation, just $R$ if $S$ is clear from context). An $R$-local system on a perfectoid space $U$ is a sheaf $\mbb{M}$ of $R_U$-modules on $U_v$ such that there is a $v$-cover $\{U_i \to U\}$ such that $\mbb{M}|_{U_i}$ is isomorphic (as a sheaf of $R_{U_i}$-modules) to $R_{U_i}^{n_i}$ for some non-negative integer $n_i$. The assignment $U \mapsto \Loc_R(U)$, the category of $R$-local systems on $U$, is evidently a $v$-stack, and we can thus evaluate $\Loc_R$ on any small $v$-sheaf $S$. The resulting category $\Loc_R(S)$ is an exact $R(S)$-linear tensor category. 

\subsubsection{$\mbb{B}_\dR^+$-lattices} If $S$ is a small $v$-sheaf, $R$ a subsheaf of $\mbb{B}_{\dR, S}$, and $\mbb{M} \in \Loc_R(S)$ then a $\mbb{B}_\dR^+$-lattice on $\mbb{M}$ is a $\mbb{B}_\dR^+$-submodule $\mc{L} \subseteq (\mbb{M} \otimes_R \mbb{B}_{\dR, S})$ such that $\mc{L} \in \Loc_{\mbb{B}_\dR^+}(S)$ and the natural map $\mc{L} \otimes_{\mbb{B}_\dR^+} \mbb{B}_{\dR, S} \to \mbb{M} \otimes_R \mbb{B}_{\dR, S}$ is an isomorphism. The most important example for us is when $R$ is the canonical inclusion of $\ul{L}$ into $\mbb{B}_\dR^{(+)}$ on $\Perfd_L$ whenever $L$ is a $p$-adic field (usually just $L = \QQ_p$). 

\subsubsection{$G$-local systems and torsors}
Fix $L$ a non-archimedean extension of $\QQ_p$ and $G$ an affine algebraic group over $L$. Let $R$ be one of the sheaves $\ul{\Lambda}$ for $\Lambda$ a topological $L$-algebra, $\mc{O},$ or $\mbb{B}_\dR^{(+)}$, which we consider as a sheaf on $\Perfd_L$. In the case, $R = \mbb{B}_\dR^{(+)}$, we make the additional assumption that $L$ is a $p$-adic field  which ensures that $R$ is a sheaf of $L$-algebras. 

Let $S$ be a small $v$-sheaf over $\Spd\+ \QQ_p$. A $G(R)$-local system on $S$ is an exact $L$-linear tensor functor $\Rep\+ G \to \Loc_R(S)$. A $G(R)$-torsor on $S$ is a sheaf $\mc{T}$ on $\Perfd_L/S$ together with a right action of the $v$-sheaf $G(R)$ (defined by $U \mapsto G(R(U))$) such that $v$-locally, $\mc{T}$ is isomorphic to $G(R)$ with right multiplication action. 

\subsection{Bundles on the Fargues--Fontaine curve}

\subsubsection{Vector bundles}
For $T$ a perfectoid space in characteristic $p$, the relative Fargues--Fontaine curve $\FF_{T}$ (denoted as $X_{T}$ in \cite{fargues-scholze:geometrization}) is an analytic adic space over $\QQ_p$, constructed as a quotient $Y_{T}/\varphi^\ZZ$ of a certain adic space $Y_{T}$ by an automorphism $\varphi$ \cite[Chapter II]{fargues-scholze:geometrization}. If $T = P^\flat$ is the tilt of a perfectoid space $P$ over $\QQ_p$, then $P$ gives distinguished point $\infty_P \in \FF_{P^\flat}(P)$ presenting $P$ as a Cartier divisor \cite[Proposition II.1.4]{fargues-scholze:geometrization} in $\FF_{P^\flat}$. We write $\FF_P$ for $\FF_{P^\flat}$ with the specified point $\infty_P$. We let $\Vect_\FF$ be the $v$-stack on $\Perfd_{\QQ_p}$ sending a perfectoid space $P$ to the category of vector bundles on $\FF_{P^\flat}$ (c.f.\ \cite[Proposition II.2.1]{fargues-scholze:geometrization}). When $P=\Spa(C, C^+)$ for $C$ an algebraically closed perfectoid field, there is a well defined notion of degree and slope on $\Vect_\FF(\Spa(C, C^+))$ leading to Harder--Narasimhan slope filtrations and the notion of semistable vector bundles \cite[Chapter II \S 2.4]{fargues-scholze:geometrization}. For a general perfectoid space $P$ and $\mc{V} \in \Vect_\FF(P)$, the function taking a geometric point $x \colon \Spa\+ C \to P$ to the Harder--Narasimhan polygon of $x^\ast \mc{V}$ is upper semicontinuous \cite[Theorem II.2.19]{fargues-scholze:geometrization}. If the Harder--Narasimhan polygon is constant, there is a global Harder--Narasimhan filtration, i.e.\ there is a decreasing, exhaustive, separated $\QQ$-filtration $\mc{V}^{\geq \lambda}, \lambda \in \QQ$ by subobjects in $\Vect_\FF(P)$, that is functorial and uniquely characterized by the fact that $\mc{V}^{\lambda} := \mc{V}^{\geq \lambda}/\mc{V}^{> \lambda}$ is semistable of slope $\lambda$ after pullback to every geometric point \cite[Theorem II.2.19]{fargues-scholze:geometrization}. For $\lambda in \mathbb{Q}$, we write $\Vect_\FF^\lambda(P)$ for the full subcategory of vector bundles on $\FF_{P^\flat}$ with constant Harder--Narasimhan polygon whose slope filtration is concentrated in degree $\lambda$ (and say these vector bundles are semistable of slope $\lambda$). The assignment $P \mapsto \Vect^\lambda_\FF(P)$ is visibly a $v$-substack of $\Vect_\FF$, so it makes sense to evaluate $\Vect_\FF^\lambda(S)$ for any small $v$-sheaf $S/\Spd\+ \QQ_p$. 

For $P$ a perfectoid space over $\breve{\QQ}_p$, there is a faithful functor 
\begin{equation}\label{eq.isoc-to-vb}
\Isoc \to \Vect_\FF(P), \qquad W \mapsto \mc{E}(W) := (W \otimes_{\breve{\QQ}_p} Y_P)/(\varphi_W \otimes \varphi)^\ZZ. 
\end{equation} 
When $P=\Spa(C, C^+)$ for $C$ an algebraically closed perfectoid field, the map \eqref{eq.isoc-to-vb} is bijective on isomorphism classes, and restricts to an equivalence of categories between the category of $D_{-\lambda}$-isotypic isocrystals and vector bundles on $\FF_{C}$ that are semistable of slope $\lambda$ \cite[Theorem II.2.14]{fargues-scholze:geometrization}. 

We let $\Vect_{\FF}^{\HNsplit}$ be the $v$-stack sending a perfectoid space $P/\Spa \mathbb{Q}_p$ to the category of vector bundles on $\FF_{P^\flat}$ with a splitting of the Harder--Narasimhan filtration. Concretely, for quasi-compact $P$, $\Vect_\FF^{\HNsplit}(P) = \bigoplus_{\lambda \in \QQ} \Vect_\FF^\lambda(P)$ is the category of $\QQ$-graded vector bundles $\mc{V} = \oplus_{\lambda \in \QQ} \mc{V}^\lambda$ such that each $\mc{V}^\lambda$ is semistable of slope $\lambda$. The functor \eqref{eq.isoc-to-vb} naturally factors through $\Vect_\FF^{\HNsplit}(P)$, and yields an equivalence $\Isoc \xrightarrow{\sim} \Vect_\FF^{\HNsplit}(\Spa(C, C^+))$ when $C$ is an algebraically closed perfectoid field.

\subsubsection{Banach--Colmez spaces}
Let $P$ be a perfectoid space over $\QQ_p$, and $\mc{E}$ a vector bundle on $\FF_{P}$. The Banach--Colmez space $\BC(\mc{E})$ is a $\ul{\mathbb{Q}_p}$-module on $\Perfd_P$ whose value on a perfectoid space $U$ over $P$ is $H^0(\FF_U, \mc{E}|_{\FF_U})$ \cite[Proposition II.2.16]{fargues-scholze:geometrization}. We also define $\BC(\mc{E}[1])$ as the $v$-sheafification of $U \mapsto H^1(\FF_U, \mc{E}|_{\FF_U})$ on $\Perfd_P$, which is also a $\ul{\mbb{Q}_p}$-module over $P$. The assignments $\mc{E} \mapsto \BC(\mc{E})$ and $\BC(\mc{E}[1])$ are morphisms of $v$-stacks from $\Vect_{\FF}$ to the stack of $\ul{\mbb{Q}_p}$-modules.  

For a small $v$-sheaf $S$ over $\Spd\+ \QQ_p$, the functor $\Vect_\FF^0(S) \to \Loc_{\QQ_p}(S)$ sending $\mc{E} \mapsto \BC(\mc{E})$ is an equivalence of categories. A quasi-inverse functor sends a local system $V \in \Loc_{\QQ_p}(S)$ to the element $V \otimes_{\ul{\QQ_p}} \mc{O}_\FF \in \Vect_\FF(S)$ \cite[Corollary II.2.20]{fargues-scholze:geometrization}.

\subsubsection{$G$-bundles}

Let $G$ be a connected linear algebraic group over $\QQ_p$, and $S$ a small $v$-sheaf over $\QQ_p$. A $G$-bundle on $\FF_S$ is an exact $\otimes$-functor $\mc{E} \colon \Rep\+ G \to \Vect_\FF(S)$, and an HN-splitting consists of a factorization of $\mc{E}$ through an exact tensor functor $\mc{E}^\gr \colon \Rep\+ G \to \Vect_\FF^{\HNsplit}(S)$. Given $b \in G(\breve{\QQ}_p)$, we get a $G$-bundle $\mc{E}_b$ given by the composition $W_b \colon \Rep\+ G \to \Isoc$ and the functor \eqref{eq.isoc-to-vb}. The $G$-bundle $\mc{E}_b$ naturally is HN-split, and we write $\mc{E}_b^\gr$ when we want to consider it as a HN-split $G$-bundle rather than a $G$-bundle (i.e. valued in $\Vect^{\HNsplit}_\FF$ versus $\Vect_\FF$). For example, the sheaves of automorphisms $\ul{\Aut}(\mc{E}_b^\gr)$ and $\ul{\Aut}(\mc{E}_b)$ are different whenever $b$ is not basic: the former is   $\ul{G_b(\QQ_p)}$ (where $G_b$ is the automorphism group of of the $G$-isocrystal $W_b$), a proper subsheaf of the latter. 

The identification $\ul{\Aut}(\mc{E}_b^\gr) = \ul{G_b(\QQ_p)}$ gives another way of constructing HN-split $G$-bundles: if $\mc{T}$ is a $\ul{G_b(\QQ_p)}$-torsor on $S$, then $\mc{T} \times^{\ul{G_b(\QQ_p)}} \mc{E}_b^\gr$ is the $G$-bundle whose value on $V \in \Rep\+ G$ is the vector bundle given by descending $\mc{E}_b^\gr(V)|_{\FF_{\mc{T}}}$ by the diagonal action of $\ul{G_b(\QQ_p)}$. 

Finally, we write $\mr{Bun}_G$ (resp.\ $\mr{Bun}_G^{\HNsplit}$) for the $v$-stack over $\Spd\, \overline{\mathbb{F}}_p$ that sends a perfectoid space $P/\overline{\mathbb{F}}_p$ to the groupoid of $G$-bundles (resp. $G$-bundles with an HN-splitting) on the relative Fargues--Fontaine curve $\FF_P$. The following theorem gives a precise description of $\mr{Bun}_G^{\HNsplit}$ (which we will principally use after base change to $\Spd\+ \Qpbreve$). 

\begin{theorem}[Proposition III.4.7 \cite{fargues-scholze:geometrization}]\label{thm:bun-HN-split}
Let $G$ be a connected linear algebraic group\footnote{Proposition III.3.7 in \cite{fargues-scholze:geometrization} is stated only for $G$ reductive, but the proof works for any connected linear algebraic group.} over $\QQ_p$. The natural map 
    \[ \bigsqcup_{[b] \in B(G)} [\ast/\ul{G_b(\QQ_p)}] \to \mr{Bun}_G^{\HNsplit} \]
sending a $\ul{G_b(\QQ_p)}$-torsor $\mc{T}$ to the HN-split $G$-bundle $\mc{T} \times^{\ul{G_b(\QQ_p)}} \mc{E}_b^\gr$ is an isomorphism of $v$-stacks over $\Spd\, \overline{\mathbb{F}}_p$. A quasi-inverse functor sends an HN-split $G$-bundle $\mc{E}^{\gr}$ whose underlying $G$-bundle is locally isomorphic to $\mc{E}_b$ to the sheaf of isomorphisms of HN-split $G$-bundles
    \[ \ul{\Isom}^\otimes(\mc{E}^\gr_b, \mc{E}^\gr). \]
\end{theorem}

\subsubsection{Modifications of bundles}
Let $P$ be an affinoid perfectoid space over $\QQ_p$, and let $\FF_P$ be the pointed Fargues--Fontaine curve relative to $P$. Recall that the point $\infty \colon P \to \FF_P$ realizes $P$ as a Cartier divisor and there is a natural isomorphism $\varprojlim_{n} \mc{O}_{\FF_P}/\mc{I}_\infty^n \cong \mbb{B}_\dR^+,$ where  $\mc{I}_\infty \subseteq \mc{O}_{\FF_P}$ is the ideal sheaf of $\infty$ and $\mbb{B}_\dR^+$ is restricted to $\Perfd_{/P}$. In particular, we get exact tensor functors 
    \[ \Vect_\FF(P) \to \Loc_{\mbb{B}_\dR^{(+)}}(P), \qquad \mc{V} \mapsto \mc{V} \otimes_{\mc{O}_{\FF_P}} \mbb{B}_\dR^{(+)}. \]
These are natural in $P$ and define restriction functors $\Vect_\FF \to \Loc_{\mbb{B}_\dR^{(+)}}$ of stacks over $\Perfd_{\mathbb{Q}_p}$, which for ease of notation we will denote as $(-)\boxtimes \mbb{B}_\dR^{(+)}$. 

Suppose now $S$ is a small $v$-sheaf over $\Spd\+ \QQ_p$ and we are given $\mc{E} \in \Vect_\FF(S)$, and $\mc{L} \subseteq \mc{E} \boxtimes \mbb{B}_\dR$ a $\mbb{B}_\dR^+$-lattice. We can form a new vector bundle $\mc{E}_{\mc{L}} \in \Vect_\FF(S)$ by modifying $\mc{E}$ at $\infty$ by $\mc{L}$ by Beauville-Laszlo gluing. (It suffices to define $\mc{E}_\mc{L}$ when $S$ is perfectoid, and the relevant gluing lemmas are in \cite{scholze:berkeley}; see the discussion of \cite[\S 4.3]{howe-klevdal:ap-pahsII}.) If $G$ is an algebraic group over $\QQ_p$ all of the above discussion extends to $G$-bundles on $\FF_S$: if $\mc{E} \colon \Rep\+ G \to \Vect_\FF(S)$ is a $G$-bundle, we write $\mc{E} \boxtimes \mbb{B}_\dR^{(+)}$ for the composition of $\mc{E}$ with the $\mbb{B}_\dR^{(+)}$-restriction functor. If $\mc{L} \subseteq \mc{E} \boxtimes \mbb{B}_\dR$ is a $\mbb{B}_\dR^+$-lattice, the modification of $\mc{E}$ by $\mc{L}$ is the $G$-bundle $\mc{E}_{\mc{L}}$ on $\FF_S$ whose value at $V \in \Rep\+ G$ is $\mc{E}_\mc{L}(V) = \mc{E}(V)_{\mc{L}(V)}$.

\section{Admissible pairs and $p$-adic Hodge structures}\label{s.admissible-pairs}
In this section, we give the general definition of an admissible pair and a $p$-adic Hodge structure on a small $v$-sheaf $S/\Spd\, \QQ_p$. We show that the categories of admissible pairs and $p$-adic Hodge structures satisfy $v$-descent, that there is an equivalence between basic admissible pairs on $S$ and $p$-adic Hodge structures on $S$ (extending the equivalence of \cite{howe-klevdal:ap-pahsI} when $S=\Spd\+ C$ for $C/\mathbb{Q}_p$ algebraically closed) and that they form Tannakian categories when $S$ is connected. We discuss the relationship between the admissible pairs/$p$-adic Hodge structures defined here and the neutral admissible pairs defined in \cite{howe-klevdal:ap-pahsII}. 
Finally, we discuss the lattice period morphisms (defined on certain pro-\'etale covers of $S$), and recall the definition of the type of a lattice. 

\subsection{Definitions and first properties}

\begin{definition}
    Let $S/\Spd\+ \mbb{Q}_p$ be a small $v$-sheaf.
    \begin{enumerate}
        \item A \emph{$p$-adic Hodge structure over $S$} is a $\mathbb{Q}$-graded $\mathbb{B}^+_\dR$-latticed $\mathbb{Q}_p$-local system $(V,\mc{L}_\dR)= \oplus_{\lambda \in \mathbb{Q}} (V_\lambda, \mathcal{L}_{\dR,\lambda})$ over $S$ such that, at each geometric point, each of the modified bundles $\mathcal{E}_\lambda(V):=(V_\lambda \otimes \mathcal{O}_{\FF})_{\mathcal{L}_{\dR,\lambda}}$ is semistable of slope $-\lambda/2$. A morphism of $p$-adic Hodge structures over $S$ is a morphism of $\mathbb{Q}$-graded $\mathbb{Q}_p$-local systems whose $\mathbb{B}_\dR$-linear extension sends one lattice into the other. 
        \item An \emph{admissible pair over $S$} is a pair $(\mc{E}, \mc{L}_\et)$ where $\mc{E}=\bigoplus \mc{E}_\lambda \in \Vect^{\HN}(S)$ is an HN-split vector bundle on the relative Fargues--Fontaine curve over $S$, and $\mc{L}_\et \subseteq \mc{E} \boxtimes \mbb{B}_\dR$ is a $\mbb{B}^+_\dR$-lattice such that $\mc{E}_{\mc{L}_\et}$ is semistable of slope zero at each geometric point. A morphism of admissible pairs is a morphism of $\mathrm{HN}$-split vector bundles whose $\mbb{B}_\dR$-linear extension sends one lattice into the other. We write $\AdmPair(S)$ for the rigid $\mbb{Q}_p$-linear tensor category of admissible pairs over $S$. An admissible pair $(\mc{E}, \mc{L}_\et)$ over $S$ is \emph{basic} if $\mc{L}_\et=\bigoplus \mc{L}_{\et,\lambda}$ for $\mc{L}_{\et, \lambda}:=\mc{E}_{\lambda}\boxtimes \mbb{B}_\dR \cap \mc{L}_\et$. 
    \end{enumerate}
    We write $\HS(S)$ (resp. $\AdmPair(S)$, resp. $\AdmPair^{\basic}(S)$) for the rigid $\mbb{Q}_p$-linear tensor category of $p$-adic Hodge structures (resp. admissible pairs, resp. basic admissible pairs) over $S$. 
\end{definition}

For $S/\Spd\+ \mbb{Q}_p$ a small $v$-sheaf and $s: \Spd\+ (C,C^+) \rightarrow S$ a geometric point, we write $\omega_{s}$ for the functor 
\[ \HS(S) \rightarrow \Vect_{\mbb{Q}_p},\; (V,\mc{L}_\dR)\mapsto H^0(\Spd\+(C,C^+), s^*V).\]
We also write $\omega_s$ for the functor 
\[ \AdmPair(S) \rightarrow \Vect_{\mbb{Q}_p},\; (\mc{E}, \mc{L}_\et) \mapsto H^0(\FF_C, s^*\mc{E}_{\mc{L}_\et}). \]
The doubled notation will not cause confusion because, under the natural functor from $\AdmPair(S)$ to $\mbb{B}^+_\dR$-latticed $\mbb{Q}_p$-local systems on $S$, 
\[ \omega_{\et}^\mc{L}: (\mc{E},\mc{L}_\et) \mapsto (\BC(\mc{E}_{\mc{L}_\et}), \mc{E} \boxtimes \mbb{B}^+_\dR), \]
there is a canonical identification of $\omega_s$ on $\AdmPair(S)$ with the composition of $\omega_{\et}^\mc{L}$ and the obvious extension of $\omega_s$ on $\HS(S)$ to $\mbb{B}^+_\dR$-latticed $\mbb{Q}_p$-local systems on $S$.

\begin{theorem}\label{theorem.ap-hs-properties}\hfill
\begin{enumerate}
    \item The assignments $S \mapsto \AdmPair(S)$, $S \mapsto \AdmPair^{\basic}(S)$, and $S\mapsto \HS(S)$ are $v$-stacks.
    \item For any small $v$-sheaf $S/\Spd\+\mbb{Q}_p$, the categories $\AdmPair(S)$, $\AdmPair^\basic(S)$, and $\HS(S)$ are rigid tensor categories over $\mbb{Q}_p$.
    \item The functor $\omega_{\et}^\mc{L}$ is an exact tensor functor and restricts to an equivalence
    \[ \AdmPair^{\mr{basic}}(S) \cong \HS(S) \]
    \item If $S/\Spd\+\mbb{Q}_p$ is a locally spatial diamond such that $\pi_0(S)$ is discrete, then $\AdmPair(S)$, $\AdmPair^\basic(S)$, and $\HS(S)$ are abelian categories. If $S$ is moreover connected and $s: \Spd\+(C,C^+) \rightarrow S$ is a geometric point then $\omega_s$ is a fiber functor on each of $\AdmPair(S)$, $\AdmPair^\basic(S)$, and $\HS(S)$ so that these categories are neutral Tannakian. 
\end{enumerate}
\end{theorem}
\begin{proof}
Both (1) and (2) are clear because $\QQ_p$-local systems/$\mbb{B}_\dR^+$-local systems/vector bundles on the relative Fargues--Fontaine curve all form $\QQ_p$-linear rigid tensor categories that satisfy $v$-descent. 

For (3), a quasi-inverse functor sends $(V, \mc{L}_\dR) = \oplus_{\lambda} (V_\lambda, \mc{L}_{\dR,\lambda})$ to 
\[ \mc{E} = \oplus_{\lambda} (V_{-2\lambda} \otimes_{\ul{\mbb{Q}}_p} \mc{O}_{\FF})_{\mc{L}_{\dR, -2\lambda}}\]
with \'etale lattice $\mc{L}_{\et} = V \otimes_{\ul{\QQ_p}}\mbb{B}_\dR^+ =\oplus_{\lambda} V_{-2\lambda} \otimes_{\ul{\QQ_p}} \mbb{B}_\dR^+$. 

For (4), to show $\AdmPair(S)$ is an abelian category, the key point is to show that the kernel of a map $f \colon (\mc{E}, \mc{L}) \to (\mc{E}', \mc{L}')$ of admissible pairs is again an admissible pair; the same then holds for cokernels by duality.  Now $\ker(f)$ has a $\QQ$-grading by $\ker(\mc{E}_\lambda \to \mc{E}'_\lambda)$, which we need to show is a vector bundle on the relative Fargues--Fontaine curve over $S$ that is semistable of slope $\lambda$. By \cite[Theorem 8.5.12]{kedlaya-liu:relative-foundations}, the category of semistable slope $\lambda = \frac{c}{d}$ vector bundles on $\FF_S$ is equivalent to the category of $\mathbb{Q}_{p^d}$-local systems on $S$ with a semilinear action of Frobenius. We thus reduce to showing that the category of $\mathbb{Q}_{p^d}$-local systems on $S$ form an abelian category when $\pi_0(S)$ is discrete. This follows from \cite[Theorem 3.24]{mann-werner} upon noticing that any $\mathbb{Q}_{p^d}$-local system admits a $\mathbb{Z}_{p^d}$-lattice on an \'etale cover \cite[Lemma 3.10]{mann-werner}. Thus $\ker(\mc{E} \to \mc{E}')$ is a vector bundle. That the modification of $\ker(f)$ by $\ker(\mc{L} \to \mc{L}')$ is semi-stable of slope $0$ at each geometric point follows from \cite[Theorem 5.1.6]{howe-klevdal:ap-pahsI} (see also \cite[Theorem 7.2]{howe-klevdal:ap-pahsII}). The arguments for $\AdmPair^\basic$ and $\HS$ are similar, and that $\omega_s$ is a fiber functor when $S$ is connected follows as above from the full-faithfulness of restriction to a point for $\mathbb{Q}_{p}$-local systems.

\end{proof}

\begin{remark}
If $\pi_0(S)$ is not discrete, then $\AdmPair(S)$ and $\HS(S)$ may not be abelian. Indeed, both contain $\mbb{Q}_p$-local systems as a full subcategory, but if $\pi_0(S)$ is not discrete, then the kernel/cokernel of a map of $\mbb{Q}_p$-local systems on $S$ is not always a $\mbb{Q}_p$-local system on $S$. For example, if $S$ is a small $v$-sheaf with $\pi_0(S) = \ZZ_p$ (e.g.\ $S = \ul{\ZZ_p}$) and $V=\ul{\mbb{Q}_p}$ is the trivial rank one local system, then $V$ admits an endomorphism whose value on a perfectoid $U \xrightarrow{f} S$ sends a section $s \in V(U)$ to $\pi_0(f) \cdot s$. The kernel of this map restricts to the trivial rank one local system on the connected component $0 \in \ZZ_p = \pi_0(S)$, but is zero on any other connected component. Thus the kernel is not a $\QQ_p$-local system on $S$; similarly, the cokernel is supported on the connected component $0 \in \ZZ_p = \pi_0(S)$ where it restricts to a rank one local system, so it is not a $\QQ_p$-local system on $S$.
\end{remark}

\begin{definition}\label{def.motivic-galois-group}
Suppose $S$ is a connected locally spatial diamond, $s \colon \Spa(C,C^+) \to S$ is a geometric point, and $A = (\mc{E}, \mc{L}_\et)$ is an admissible pair on $S$. The \emph{motivic Galois group} $\MG(A, s)$ is the algebraic group over $\QQ_p$ representing the automorphisms of $\omega_s$ restricted to the Tannakian subcategory $\langle A \rangle \subseteq \AdmPair(S)$ generated by $A$. 
\end{definition}

We consider the following natural exact tensor functors:

\begin{align}
\begin{split}
    &\omega_\et: \AdmPair \rightarrow \Loc_{\ul{\mbb{Q}_p}},\;  (\mc{E}, \mathcal{L}_\et) \mapsto \BC(\mc{E}_{\mathcal{L}_\et}) \\
    &\omega_\et: \HS \rightarrow \Loc_{\ul{\mbb{Q}_p}},\; (V, \mc{L}_\dR) \mapsto V \\
    &\omega_\FF: \AdmPair \rightarrow \Vect^{\HNsplit}_{\FF},\; (\mc{E}, \mathcal{L}_\et) \mapsto \mc{E} \\
    &\omega_\FF: \HS \rightarrow \Vect^{\HNsplit}_{\FF},\; \oplus_\lambda (V_\lambda, \mathcal{L}_\lambda) \mapsto \oplus_{\lambda} (V_\lambda \otimes_{\ul{\mbb{Q}_p}} \mathcal{O}_\FF)_{\mathcal{L}_\lambda}   \\
    &\omega_{\mc{L}_\et}: \AdmPair \to \Loc_{\mbb{B}_\dR^+},\; (\mc{E}, \mc{L}_{\et}) \mapsto \mc{L}_\et\\
    &\omega_{\mc{L}_\et}: \HS \to \Loc_{\mbb{B}_\dR^+},\; (V, \mc{L}_\dR) \mapsto V \otimes_{\ul{\mbb{Q}_p}} \mbb{B}^+_\dR \\
    &\omega_{\mc{L}_\dR}: \AdmPair \to \Loc_{\mbb{B}_\dR^+},\;  (\mc{E}, \mc{L}_\et) \mapsto \mc{E} \boxtimes \mbb{B}_\dR^+ \\
    &\omega_{\mc{L}_\dR}: \HS \to \Loc_{\mbb{B}_\dR^+},\;  (V, \mc{L}_\dR) \mapsto \mc{L}_\dR  
\end{split}
\end{align}

For $\bullet = \et, \FF, \mc{L}_\et, \mc{L}_\dR$, there are canonical isomorphisms $\omega_\bullet \circ \omega_{\et}^\mathcal{L}=\omega_\bullet$ of functors on $\AdmPair^\basic$; and moreover we have canonical identifications 
    \[ \omega_{\mc{L}_\et} \otimes_{\mbb{B}^+_\dR} \mbb{B}_\dR = \omega_\FF \boxtimes \mbb{B}_\dR = \omega_{\mc{L}_\dR} \otimes_{\mbb{B}^+_\dR} \mbb{B}_\dR. \]

\subsubsection{Relation with neutral $p$-adic Hodge structures and admissible pairs.}\label{ss.neutral-admissible-pairs} We discuss here the relation between the general theory developed here and the neutral theory of \cite{howe-klevdal:ap-pahsII}. Recall that a neutral $p$-adic Hodge structure on $S$ is a finite dimensional $\mbb{Q}$-graded pair $\bigoplus_{\lambda} (V_{\lambda},\mc{L}_{\dR,\lambda})$ where $V_{\lambda}$ is a $\mbb{Q}_p$-vector space and $\mc{L}_{\dR,\lambda}$ is a $\mbb{B}^+_\dR$ lattice in $V_{\lambda}\otimes_{\mbb{Q}_p} \mbb{B}_\dR$ such that $(V_{\lambda} \otimes_{\mbb{Q}_p} \mc{O}_\FF)_{\mc{L}_\dR}$ is semistable of slope $-\lambda/2$ at each geometric point. There is a natural functor from neutral $p$-adic Hodge structures on $S$ to $p$-adic Hodge structures on $S$ given by 
\begin{equation}\label{eq.neutral-hs-to-hs} \bigoplus (V_{\lambda},\mc{L}_{\dR,\lambda}) \mapsto \bigoplus( \ul{V_\lambda}, \mc{L}_{\dR,\lambda}). \end{equation}

Recall that a neutral admissible pair on $S$ is a pair $(W,\mc{L}_\et)$ where $W$ is an isocrystal and $\mc{L} \subseteq W \otimes \mbb{B}_\dR$ is a $\mbb{B}^+_\dR$-lattice such that $\mc{E}(W)_{\mc{L}}$ is semistable of slope zero at every geometric point. There is a natural functor from neutral admissible pairs to admissible pairs on $S$
\begin{equation}\label{eq.neutral-ap-to-ap} (W, \mc{L}_{\et}) \mapsto ( \mc{E}^\HN(W), \mc{L}_{\et}). \end{equation}

The following is then immediate from the definitions and the discussion in the proof of \cref{theorem.ap-hs-properties}. 
\begin{proposition}
    The functors \cref{eq.neutral-ap-to-ap} and \cref{eq.neutral-hs-to-hs} are faithful exact tensor functors. If $S$ is connected, they are fully faithful. 
\end{proposition}

In particular, if $S$ is a connected locally spatial diamond, the motivic Galois group of a neutral admissible pair as defined in \cite{howe-klevdal:ap-pahsII} agrees with the motivic Galois group of the associated admissible pair as defined in \cref{def.motivic-galois-group}. 

\subsection{$G$-structure for admissible pairs and $p$-adic Hodge structures}\label{ss.G-structure}

\begin{definition}
Let $S/\Spd\+ \mbb{Q}_p$ be a small $v$-sheaf and let $G/\mbb{Q}_p$ be a connected linear algebraic group. 
\begin{enumerate}
    \item A $G$-$p$-adic Hodge structure on $S$ is an exact tensor functor $\mc{G}: \Rep\, G \rightarrow \HS(S)$ such that, for any geometric point $s: \Spd(C,C^+)\rightarrow S$, $\omega_s \circ \mc{G}$ is isomorphic to $\omega_\std$. 
    
    \item A $G$-admissible pair on $S$ is an exact tensor functor $\mc{G}: \Rep\, G \rightarrow \AdmPair(S)$ such that, for any geometric point $s: \Spd(C,C^+) \rightarrow S$, $\omega_s \circ \mc{G}$ is isomorphic to $\omega_\std$.
\end{enumerate}
 If $S/\Spd\Qpbreve$, a $G$-admissible pair $\mc{G}$ on $S$ is called neutral if $\omega_\FF \circ \mc{G} \cong \mc{E}_b$ for some $b \in G(\breve{\QQ}_p)$, and an isomorphism $\omega_\FF \circ \mc{G} \cong \mc{E}_b$ of HN-graded vector bundles is called a $b$-rigidification. 
\end{definition}

Because of the equivalence between basic admissible pairs and $p$-adic Hodge structures on $S$ of \cref{theorem.ap-hs-properties}-(3), one can consider a $G$-$p$-adic Hodge structure instead as a $G$-admissible pair that factors through $\AdmPair^\basic$. Thus, we focus on admissible pairs for the remainder of this section. 

\begin{example}
If $S$ is a connected locally spatial diamond, $s \colon \Spd(C, C^+) \to S$ is a geometric point, and $A$ is an admissible pair on $S$, then $A$ admits a canonical $\MG(A, s)$-structure (see \cite[\S 2.1.3]{howe-klevdal:ap-pahsI}). 
\end{example}

\subsection{Lattice period maps}

\subsubsection{}
Let $S/\Spd\+ \mbb{Q}_p$ be a small $v$-sheaf, let $G/\mbb{Q}_p$ be a connected linear algebraic group, and let $\mc{G}$ be a $G$-admissible pair on $S$. We define the $v$-sheaf of \'etale trivializations of $\mc{G}$
\[ \tilde{S}_{\mc{G},\et}:= \ul{\Isom}^\otimes(\omega_\std \otimes_{\mbb{Q}_p} {\ul{\mbb{Q}_p}}, \omega_\et \circ \mc{G}) \]

Concretely, if $T \in \Perfd_{\mbb{Q}_p}$ is a perfectoid space, then $\tilde{S}_{\mc{G},\et}(T)$ classifies pairs consisting of a map $T \to S$ and an isomorphism $\rho_\et \colon \omega_\std \otimes_{\mbb{Q}_p} \ul{\QQ_p} \xrightarrow{\sim} \omega_\et\circ \mc{G}|_{T}$ of $\otimes$-functors. There is a right action of $\ul{G(\QQ_p)}=G(\ul{\mbb{Q}_p})$ on $\tilde{S}_{\mc{G},\et}$ by automorphisms of $\omega_\std \otimes_{\mbb{Q}_p} \ul{\mbb{Q}_p}$. 

\subsubsection{}
Suppose $S/\Spd\+ \Qpbreve$. Writing $B(G)^{\mr{disc}}$ for the set $B(G)$ equipped with the discrete topology, there is a continuous isocrystal classifying map associated to $\mc{G}$,
\begin{equation}\label{eqn:isocrystal-classifying-map}
    |S| \to |\mr{Bun}_G^{\HN}| = B(G)^{\mr{disc}} 
\end{equation}
where the first map is induced by $\omega_\FF \circ \mc{G}$ and the second equality is from \cite[Proposition III.4.7]{fargues-scholze:geometrization} (this result is a priori for $G$ reductive but the proof goes through for $G$ a connected linear algebraic group by using \cite[Theorem 4.3]{howe-klevdal:ap-pahsII} in place of \cite[Theorem III.2.4]{fargues-scholze:geometrization}). 

For $[b] \in B(G)$ we let $S_{\mc{G},[b]}$ be the preimage of $[b]$ under the isocrystal classifying map \cref{eqn:isocrystal-classifying-map}. For $b \in G(\Qpbreve)$, we define the $v$-sheaf of $b$-rigidifications of $\mc{G}$ 
\[ \tilde{S}_{\mc{G},b} := \ul{\Isom}^\otimes(\mc{E}_b^\gr, \omega_\FF \circ \mc{G}). \]

Concretely, if $T \in \Perfd_{\mbb{Q}_p}$, $\tilde{S}_{\mc{G},b}(T)$ consists of $T \to S$ with $\rho_b \colon \mc{E}_b \xrightarrow{\sim} \omega_\FF \circ \mc{G}|_T$ (we emphasize that these functors are taken to be valued in $\Vect^{\HNsplit}(\FF_T)$, not $\Vect(\FF_T)$). There is a right action of $\ul{G_b(\QQ_p)}=G_b(\ul{\mbb{Q}_p})$ on $\tilde{S}_{\mc{G},b}$ by automorphisms of $\mc{E}_b$. 

\begin{lemma}\hfill
\begin{enumerate}
    \item Let $S/\Spd\+ \mbb{Q}_p$ be a small $v$-sheaf and let $\mc{G}$ be a $G$-admissible pair on $S$. The $v$-sheaf $\tilde{S}_{\mc{G},\et}$ is a $G(\QQ_p)$-torsor over $S$.  
    \item Let $S/\Spd\+ \Qpbreve$ be a small $v$-sheaf and let $\mc{G}$ be a $G$-admissible pair on $S$. The $v$-sheaf $\tilde{S}_{\mc{G},b}$ is a $G_b(\QQ_p)$-torsor over $S_{\mc{G},b}$. 
\end{enumerate}
In particular, if $S$ in (1) or (2) is a locally spatial diamond, then so is $\tilde{S}_{\mc{G},\et}$ or $\tilde{S}_{\mc{G},b}$. 
\end{lemma}

\begin{proof}
Parts (1) and (2) follow from the $b = 1$ and $b = b$ cases of \cref{thm:bun-HN-split}, respectively. The claim that  $\tilde{S}_{\mc{G},\et}$ or $\tilde{S}_{\mc{G},b}$ is a locally spatial diamond when $S$ is a locally spatial diamond follows from \cite[Lemma 10.13]{scholze:etale-cohomology-of-diamonds}.
\end{proof}

\subsubsection{}
Suppose $S/\Spd\+ \mbb{Q}_p$ is a small $v$-sheaf and $\mc{G}$ is a $G$-admissible pair on $S$.  Recalling that $\Gr_G$ parameterizes $\mbb{B}^+_\dR$-lattices on the trivial $G(\mbb{B}_\dR)$-local system, we define a de Rham lattice period map $\pi_{\mc{L}_\dR}: \tilde{S}_{\mc{G},\et} \rightarrow \Gr_G$ by 
\[ \pi_{\mc{L}_\dR}(\rho_\et) = \rho_\et^{-1}(\omega_{\mc{L}_\dR} \circ \mc{G}), \]
where the right-hand side is viewed as a $\mbb{B}^+_\dR$-lattice in $\omega_\std \otimes_{\mbb{Q}_p} \mbb{B}_\dR$. 

\subsubsection{}Suppose $S/\Spd\+ \Qpbreve$ is a small $v$-sheaf and $\mc{G}$ is a $G$-admissible pair on $S$.  For $b \in G(\Qpbreve)$, we define an \'{e}tale lattice period map $\pi_{\mc{L}_\et}: \tilde{S}_{\mc{G},b} \rightarrow \Gr_G$ by
    \[  \pi_{\mc{L}_\et}(\rho_b) = \rho_{b}^{-1}(\omega_{\mc{L}_\et} \circ \mc{G}), \]
where the right-hand side is viewed as a $\mbb{B}^+_\dR$-lattice in $\omega_\std \otimes_{\mbb{Q}_p} \mbb{B}_\dR$  using the canonical identification $\mc{E}_b \boxtimes \mbb{B}_\dR^+ = \omega_\std \otimes_{\QQ_p} \mbb{B}_\dR^+$. It is immediate that $\pi_{\mc{L}_\et}$ factors through the open locus $\Gr_G^{b-\adm}$ defined in \cite[\S 7.3]{howe-klevdal:ap-pahsII}. 

\subsubsection{} Let  $b \in G(\Qpbreve)$. Recall that in \cite[\S 8.1, \S 8.3]{howe-klevdal:ap-pahsII}, we have defined a moduli space $\mc{M}_b/\Spd\+ \Qpbreve$ of neutral admissible pairs equipped with de Rham and \'{e}tale lattice period maps $\pi_{\mc{L}_\et}: \mc{M}_b \rightarrow \Gr_G^{b-\adm} \subseteq \Gr_G$ and $\pi_{\mc{L}_\dR}: \mc{M}_b \rightarrow \Gr_G$. For $S/\Spd\+ \Qpbreve$ a small $v$-sheaf and $\mc{G}$ a $G$-admissible pair on $S$, we have a natural classifying map $\tilde{S}_{\mc{G}, b,\et}=\tilde{S}_{\mc{G},\et,b} \rightarrow \mc{M}_b$ fitting into a commutative diagram of $v$-sheaves:
\begin{equation}\label{fig:lattice-period-maps}
\begin{tikzcd}
	& {\tilde{S}_{\mc{G}, b, \et}} && {\tilde{S}_{\mc{G}, \et, b}} \\
	{\tilde{S}_{\mc{G},b}} && {\mc{M}_b} && {\tilde{S}_{\mc{G},\et}} \\
	& {\Gr_G^{\badm}} && {\Gr_G} \\
	& {S_{\mc{G},[b]}} & S & S
	\arrow["{=}"{description}, no head, from=1-2, to=1-4]
	\arrow[from=1-2, to=2-1]
	\arrow[from=1-2, to=2-3]
	\arrow[from=1-4, to=2-3]
	\arrow[from=1-4, to=2-5]
	\arrow["{\pi_{\mc{L}_\et}}"{description}, from=2-1, to=3-2]
	\arrow[curve={height=18pt}, from=2-1, to=4-2]
	\arrow["{\pi_{\mc{L}_\et}}"{description}, from=2-3, to=3-2]
	\arrow["{\pi_{\mc{L}_\dR}}"{description}, from=2-3, to=3-4]
	\arrow["{\pi_{\mc{L}_\dR}}"{description}, from=2-5, to=3-4]
	\arrow[curve={height=-12pt}, from=2-5, to=4-4]
	\arrow[hook, from=4-2, to=4-3]
	\arrow["{\mr{Id}}", from=4-4, to=4-3]
\end{tikzcd}
\end{equation}
The two (shape-theoretic) diamonds at the top of the diagram are Cartesian.

\subsection{Types and filtrations}
In this subsection, we define the type of an admissible pair and associated notions. We use freely definitions from \cite[\S 3]{howe-klevdal:ap-pahsII}, in particular the notion of lattices/filtrations over diamonds. 

\begin{definition}
Let $G/\mathbb{Q}_p$ be a connected linear algebraic group. 
\begin{enumerate}
    \item Let $S/\Spd\+ \Qpbreve$ be a small $v$-sheaf and $\mc{G}$ a $G$-admissible pair on $S$. For $[b] \in B(G)$, we say $\mc{G}$ is $[b]$-pure if $S_{\mc{G},[b]} = S_{[b]}$ (equivalently, if the pullback of $\omega_\FF \circ \mc{G}$ to any geometric point is isomorphic to $\mc{E}_b$ for any representative $b$ of $[b]$). 
    \item Let $[\mu]$ be a conjugacy class of cocharacters of $G_{\ol{\QQ}_p}$ with reflex field $\mathbb{Q}_p([\mu])$. For $S/\Spd\+ \Qp([\mu])$ a small $v$-sheaf and $\mc{G}$ a $G$-admissible pair on $S$, we say that $\mc{G}$ is of type $[\mu]$ if $\pi_{\mc{L}_\dR}: \tilde{S}_{\mc{G},\et} \rightarrow \Gr_G$ factors through $\Gr_{[\mu^{-1}]}$. 
\end{enumerate}
\end{definition}

Recall the subset $B(G, [\mu^{-1}]) \subseteq B(G)$, defined by \cite[\S 6]{kottwitz:isocrystals-II} for $G$ reductive, and in general by pullback via $B(G) \cong B(G/R_u G)$ for $R_u G$ the unipotent radical of $G$ (see \cite[\S 2.2.3]{howe-klevdal:ap-pahsI}). 

\begin{lemma}
Suppose $\mc{G}$ is $b$-pure. Then the following conditions are equivalent. 
\begin{enumerate}
    \item $\mc{G}$ has type $[\mu]$.
    \item The image of $\pi_{\mc{L}_\et}$ is contained in $\Gr_{[\mu]}$. 
\end{enumerate}
Moreover, if these conditions are satisfied and $S\neq \emptyset$, then $[b] \in B(G, [\mu^{-1}])$. 
\end{lemma}
\begin{proof}
Since $\mc{G}$ is $b$-pure, the map $\tilde{S}_{\mc{G}, \et, b} \to \tilde{S}_{\mc{G}, \et}$ of \cref{fig:lattice-period-maps} is surjective. Thus $\mc{G}$ has type $[\mu]$ if and only if the  $\tilde{S}_{\mc{G},\et, b} \to \mc{M}_b$ factors through 
$\mc{M}_{b, [\mu]}$, which happens if and only if the $\pi_{\mc{L}_\et}$ factors through $\Gr_{[\mu]}$ (since $\tilde{S}_{\mc{G}, b, \et} \to \tilde{S}_{\mc{G}, b}$ is always surjective). 

Suppose now that $\mc{G}$ has type $[\mu]$ and is $b$-pure. Then $\Gr_{[\mu]}^{\badm}$ is non-empty as it contains $\pi_{\mc{L}_\et}(\tilde{S}_b)$, and consequently $[b] \in B(G, [\mu^{-1}])$ by \cite[Proposition 7.10]{howe-klevdal:ap-pahsII}. 
\end{proof}

\subsubsection{} We recall the trace filtration: if $(\mc{L}_1, \mc{L}_2)$ are two $\mbb{B}_\dR^+$-lattices on a common $\mbb{B}_\dR$-local system on $S$, then 
\begin{align*}
    \trFil_{\mc{L}_1}^i (\mc{L}_2 \otimes_{\mbb{B}_\dR^+} \mc{O}) &:= \text{Image of } \l(\Fil^i \mbb{B}_\dR \cdot \mc{L}_1\r) \cap \mc{L}_2 \text{ in } \mc{L}_2 \otimes_{\mbb{B}_\dR^+} \mc{O}_S \\
    &\cong (\Fil^i \mbb{B}_\dR \cdot \mc{L}_1) \cap \mc{L}_2/(\Fil^i \mbb{B}_\dR \cdot \mc{L}_1) \cap (\Fil^1 \mbb{B}_\dR\cdot \mc{L}_2). 
\end{align*}

For an admissible pair $A = (\mc{E}, \mc{L}_\et)$ on $S$, the filtration $\trFil_{\mc{L}_\et}^\bullet\l( (\mc{E} \boxtimes \mbb{B}_\dR^+) \otimes \mc{O}\r)$ is called the Hodge filtration (written $\Fil_\Hdg^\bullet$), while the filtration $\trFil^\bullet_{\mc{E} \boxtimes \mbb{B}_\dR^+}(\mc{L}_\et \otimes_{\mbb{B}^+_\dR} \mc{O})$ is called the Hodge-Tate filtration (written $\Fil_\HT^\bullet$). Moreover, for $\mc{G}$ a $G$-admissible pair, we define Hodge and Hodge-Tate filtrations $\Fil_\Hdg^\bullet \mc{G}$ and $\Fil_\HT^\bullet \mc{G}$ on the $G(\mc{O})$-local systems $(\omega_{\FF} \circ \mc{G}) \boxtimes \mc{O}$ and $(\omega_{\mc{L}_\et} \circ \mc{G}) \otimes_{\mbb{B}_\dR^+} \mc{O}$ respectively by declaring that for $V \in \Rep\+ G$ 
    \[ \Fil_\Hdg^\bullet \mc{G}(V) = \Fil^\bullet_{\Hdg} (\omega_{\FF} \circ \mc{G}(V) \boxtimes \mc{O}_S) \textrm{ and }  \Fil_\HT^\bullet \mc{G}(V) = \Fil^\bullet_{\HT} (\omega_{\mc{L}_\et} \circ \mc{G}(V) \otimes_{\mbb{B}^+_\dR} \mc{O}). \]
    
In general, the graded pieces are $\mc{O}$-modules, rather than $\mc{O}$-local systems, and furthermore the functors may not be exact (as functors from $\Rep G$ to the exact category of filtered $\mc{O}$-modules). When the graded pieces of such a filtration are $\mc{O}$-local systems and the functors are exact, it is called a filtered $G(\mc{O})$-local system on $S$ as in \cite[\S3.2]{howe-klevdal:ap-pahsII}. 

\begin{proposition}
Let $G/\mathbb{Q}_p$ be a connected linear algebraic group, let $[\mu]$ be a conjugacy class of cocharacters of $G_{\overline{\mbb{Q}}_p}$, and let $S/\Spd\+ \Qpbreve([\mu])$ be a small $v$-sheaf. If $\mc{G}$ is a $G$-admissible pair on $S$  then the following are equivalent:
\begin{enumerate}
    \item The Hodge filtration $\Fil_\Hdg^\bullet \mc{G}$ is a filtered $G(\mc{O})$-local system of type $[\mu^{-1}]$
    \item The Hodge-Tate filtration $\Fil_\HT^\bullet \mc{G}$ is a filtered $G(\mc{O})$-local system of type $[\mu]$
    \item $\mc{G}$ of type $[\mu]$
\end{enumerate}
\end{proposition}
\begin{proof}
This follows from \cite[Theorem 3.5]{howe-klevdal:ap-pahsII} which is stated when $S$ is a diamond, but whose proof applies to any small $v$-sheaf (note that (1) if and only if (2) is straightforward, and that the main ingredient in showing that (3) implies (1) and (2) is \cite[Proposition 19.4.2]{scholze:berkeley}).
\end{proof}

\subsubsection{}\label{sss.filtration-period-maps} In particular, when $\mc{G}$ is of type $[\mu]$, we have a Hodge-Tate period map $\pi_\HT: \tilde{S}_{\mc{G},\et} \rightarrow \Fl_{[\mu]}$ classifying $\Fil^\bullet_\HT \mc{G}$ (viewed as filtration on the trivial $G(\mc{O})$-torsor via $\rho_\et \otimes_{\ul{\mbb{Q}_p}}\mc{O}$); if it is moreover $[b]$-pure, then we have a Hodge period map $\pi_{\Hdg}: \tilde{S}_{\mc{G},b} \rightarrow \Fl_{[\mu^{-1}]}$ classiyfing $\Fil^\bullet_\Hdg \mc{G}$ (viewed as filtration on the trivial $G(\mc{O})$-torsor via $\rho_b \boxtimes \mc{O}$). 
For $\BB \colon \Gr_{[\mu^{\pm 1}]} \to \Fl_{[\mu^{\mp 1}]}$ the Bialynicki-Birula map, we can extend the commutative diagram \cref{fig:lattice-period-maps} with filtration period maps:
\[\begin{tikzcd}
	{\tilde{S}_{\mc{G},b}} && {\mc{M}_{b, [\mu]}} && {\tilde{S}_{\mc{G},\et}} \\
	& {\Gr_{[\mu]}^{\badm}} && {\Gr_{[\mu^{-1}]}} \\
	& {\Fl_{[\mu^{-1}]}} && {\Fl_{[\mu]}}
	\arrow["{\pi_{\mc{L}_\et}}"{description}, from=1-1, to=2-2]
	\arrow["{\pi_{\Hdg}}"{description}, curve={height=12pt}, from=1-1, to=3-2]
	\arrow["{\pi_{\mc{L}_\et}}"{description}, from=1-3, to=2-2]
	\arrow["{\pi_{\mc{L}_\dR}}"{description}, from=1-3, to=2-4]
	\arrow["{\pi_\Hdg}"{description}, curve={height=-6pt}, from=1-3, to=3-2]
	\arrow["{\pi_\HT}"{description}, curve={height=6pt}, from=1-3, to=3-4]
	\arrow["{\pi_{\mc{L}_\dR}}"{description}, from=1-5, to=2-4]
	\arrow["{\pi_\HT}"{description}, curve={height=-12pt}, from=1-5, to=3-4]
	\arrow["\BB"{description}, from=2-2, to=3-2]
	\arrow["\BB"{description}, from=2-4, to=3-4]
\end{tikzcd}\]

\section{Admissible pairs of potential good reduction and variations of $p$-adic Hodge structure}\label{s.variations-padic-hs}

Let $L$ be a $p$-adic field and let $X/L$ be a smooth rigid analytic variety. Like the notion of a continuous family of Hodge structures on a complex manifold, the definition of a basic admissible pair/$p$-adic Hodge structure on $X^\diamond$ is too general when one is primarily interested in geometric families. Given the additional structure that rigid analytic varieties have (e.g.\ a robust differential theory), one may wonder about the possibility of refining the coarse notion of a basic admissible pair/$p$-adic Hodge structure on $X^\diamond$ of \cref{s.admissible-pairs} into something more akin to a complex variation of Hodge structure. Moreover, one would like to clarify the relation of basic admissible pairs/$p$-adic Hodge structures to well established notions in relative $p$-adic Hodge theory such as de Rham local systems. 

This section seeks to provide an answer to these questions. First, we define the notion of an admissible pair with \emph{potential good reduction} (\cref{defn:good-reduction-adm-pair}). This is an admissible pair on $X^\diamond$ that is neutralized on a potentially unramified pro-\'etale cover $\tilde{X}^\diamond$ of $X^\diamond$. Simple examples show that $\tilde{X}^\diamond$ need not be rigid analytic, but nevertheless we expect that it behaves in many ways as if it were. Namely, we expect that
\begin{enumerate}
    \item (\cref{conj.pot-unram-prop}) $\tilde{X}^\diamond$ is the diamond associated to an adic space $\hat{\tilde{X}}$ that has a good theory of differentials and, moreover, maps from $\tilde{X}^\diamond$ to $\Gr_{[\mu]}$ are in bijection with maps from $\hat{\tilde{X}}$ to the flag variety $\Fl_{[\mu^{-1}]}$ that satisfy Griffiths transversality. 
    \item (\cref{conj.transcendence}) Neutral admissible pairs on $\tilde{X}^\diamond$ have similar transcendence properties to those on rigid analytic diamonds: if $A$ is an admissible pair on $\tilde{X}^\diamond$ then the image of the Hodge-Tate period map is Zariski dense in the flag variety for the motivic Galois group of $A$. 
\end{enumerate}
For the link to de Rham local systems, we introduce the notion of a \emph{variation of $p$-adic Hodge structure} on $X$ (\cref{def.vpahs}, \cite[\S 10.1]{howe-klevdal:ap-pahsII}). This is a $\QQ$-graded de Rham local system $V = \oplus_{\lambda} V_\lambda$ on $X$ such that the canonical lattice $\mbb{M}_0(V) = D_\dR(V) \otimes_{\mc{O}_{X}} \mbb{B}_\dR^+$ (naturally $\QQ$-graded) given by applying the relative de Rham functor of \cite[Theorem 1.5]{liu-zhu:riemann-hilbert} naturally yields a $p$-adic Hodge structure on $X^\diamond$. We have the following expectation
\begin{enumerate}[resume]
    \item (\cref{conj.pot-good-red}) The equivalence between basic admissible pairs and $p$-adic Hodge structures on $X^\diamond$ restricts to an equivalence between basic admissible pairs with potential good reduction and variations of $p$-adic Hodge structure on $X$. 
\end{enumerate}

Finally, we comment that Conjectures \ref{conj.pot-unram-prop} and \ref{conj.pot-good-red} are motivated by the expectation/desire that variations of $p$-adic Hodge structure behave completely analogously to variations of Hodge structure on a complex manifold. Namely --- if we assume these conjectures --- given a variation of $p$-adic Hodge structure $V$ on a geometrically connected rigid analytic variety $X$, the associated basic admissible pair on $X$ is neutral when pulled back to the cover $\hat{\tilde{X}}$. The Hodge period map of this neutral admissible pair yields a morphism
    \[ \pi_\Hdg \colon \hat{\tilde{X}} \to \Fl_{[\mu^{-1}]}, \]
which is Griffiths transverse since it arises by reduction from the \'etale lattice period map $\pi_{\mc{L}_\et} \colon \tilde{X}^\diamond \to \Gr_{[\mu]}$. Moreover, the derivative $d\pi_{\Hdg}$ of the Hodge period map is the Kodaira-Spencer morphism from the connection on the associated $G$-bundle $D_\dR(V)$. In \cref{ss.inscribed-period-maps}, we show how a version of the above results can be obtained \emph{unconditionally} using the differential theory of \cite{Howe.InscriptionTwistorsAndPAdicPeriods}, which we view as providing reasonable evidence for Conjectures \ref{conj.pot-unram-prop} and \ref{conj.pot-good-red}.

\subsection{Potentially unramified towers}

\begin{definition}
{(\cite[Definition 3.3]{hubner:adic-tame-site}.)}
A map $f \colon Y \to X$ of {strongly Noetherian} adic spaces is \emph{strongly \'{e}tale} at $y \in |Y|$ if it is \'etale at $y$ and the extension of residue fields $\kappa(y)/\kappa(f(y))$ is unramified with respect to the valuations $|\cdot(y)|$ and $|\cdot(f(y))|$. A map $f \colon Y \to X$ is strongly \'etale if it is strongly \'etale at every point $y \in |Y|$.
\end{definition}

Since we are working with arbitrary valuations, we recall what it means for a finite extension $E/F$ of valued fields to be unramified. Let $E^+, F^+$ be the valuation rings of $E, F$ respectively, so that $F^+ = E^+ \cap F$. Choose $\ol{E}$ an algebraic closure of $E$, and let $E^{+,\mr{sh}} \subseteq \ol{E}$ be the strict henselization of $E^+$ \cite[\href{https://stacks.math.columbia.edu/tag/04GQ}{Tag 04GQ}]{stacks-project}. Note $E^{+, \mr{sh}}$ can be computed as $\colim_{(R, \mf{q})} R$ where the indices range over subrings $E^+ \subseteq R \subseteq \ol{E}$ such that $E^+ \subseteq R$ is \'etale and $\mf{q} \cap E^+$ is the maximal ideal of $E^+$. We let $E^{\mr{sh}} = \mr{Frac}(E^{+,\mr{sh}})$, and write $F^{+, \mr{sh}}, F^{\mr{sh}}$ for the strict henselization of $F^+$ and its fraction field, also computed as a subring of $\ol{E}$. Note that there is always an inclusion $F^{+, \mr{sh}} \subseteq E^{+, \mr{sh}}$ since if $R$ appears in the colimit for $F^{+, \mr{sh}}$, then the image of $R \to R \otimes_{F^+} E^+ \to \ol{E}$ is an \'etale $E$-subalgebra contained in $E^{+,\mr{sh}}$.  The extension $E/F$ of valued fields is (by definition) unramified if $F^{\mr{sh}} = E^{\mr{sh}}$ as subrings of $\ol{E}$. This definition is independent of the algebraic closure $\ol{E}$ by \cite[\href{https://stacks.math.columbia.edu/tag/04GU}{Tag 04GU}]{stacks-project}. 

As an immediate consequence of the above discussion, for a tower $M/E/F$ of finite extensions of valued fields, the extension $M/F$ is unramified if and only if both $M/E$ and $E/F$ are unramified. This immediately yields the following.
\begin{lemma}\label{lem.composition-strongly-etale}
Let $f \colon Y \to X$ and $g \colon Z \to Y$ be \'etale maps of strongly Noetherian adic spaces and $z \in |Z|$. Then $g \circ f$ is strongly \'etale at $z$ if and only if $g$ is strongly \'etale at $z$ and $f$ is strongly \'etale at $g(z)$.     
\end{lemma}

\begin{definition}\label{defn.potentially-unramified}
Let $L$ be a $p$-adic field, let $X/L$ be a rigid analytic variety, and let $\tilde{X} \in X_\proet$. 
\begin{enumerate}
    \item We say $\tilde{X}$ is \emph{unramified} if it admits a pro-\'{e}tale presentation $\tilde{X} \cong (X_i)_{i \in I}$ where each map $X_i \to X$ is strongly \'etale (in particular, the transition maps are strongly \'etale).
    \item We say $\tilde{X}$ is \emph{potentially unramified} if there is an \'{e}tale cover $\{U_j \rightarrow X\}$ such that, for each index $j$, $\tilde{X} \times_X U_j \rightarrow U_j$ is unramified (as an element of $U_{j, \proet}$).
\end{enumerate}
\end{definition}

We record a stability properties of (potentially) unramified objects.

\begin{lemma}\label{lemma.pur-product}
    If $\tilde{X}$ and $\tilde{Y}$ are (potentially) unramified objects in $X_\proet$, then $\tilde{X} \times_X \tilde{Y}$ is a (potentially) unramified object in $X_\proet$. 
\end{lemma}
\begin{proof}
    This follows since the product of two strongly \'{e}tale maps is strongly \'{e}tale (\cite[Top of p.885]{hubner:adic-tame-site}). 
\end{proof}

Recall that any $\tilde{X} \in X_\proet$ admits an underlying topological space $|\tilde{X}|=\lim_{i \in I} |X_i|$ and an associated diamond $\tilde{X}^\diamond=\lim_{i \in I} X_{i}^\diamond$ such that $|\tilde{X}^\diamond|=|\tilde{X}|.$ A potentially unramified $\tilde{X}$ is essentially the opposite of a perfectoid object --- instead of admitting enough ramification at $p$ to become a perfectoid space after completion, it admits so little ramification at $p$ that we expect it to still behave similarly to a rigid analytic variety after completion. 

As a first example of this phenomenon, we show an analog of the density of classical points in a rigid analytic variety. For $L$ a $p$-adic field,  
recall from \cite[Definition 2.1]{howe-klevdal:ap-pahsII} that a rigid analytic point of $Y/\Spd\+ L$ is a map $\Spd\+ L' \rightarrow \tilde{X}^\diamond$ for $L'/L$ a finite extension.

\begin{lemma}\label{lemma.pur-density}
Let $L$ be a strict $p$-adic field (i.e. a $p$-adic field with algebraically closed residue field). If $X/L$ is a rigid analytic variety and $\tilde{X} \in X_\proet$ is a potentially unramified tower, then (the images of) rigid analytic points are dense in $|\tilde{X}|=|\tilde{X}^\diamond|.$
\end{lemma}
\begin{proof}
We can replace $X$ by an \'etale cover, and assume that $\tilde{X}$ is unramified. Let $(X_i)_{i \in I}$ be a presentation of $\tilde{X}$ with each $X_j \to X$ strongly \'etale. Any open in $|\tilde{X}| = \varprojlim_i |X_i|$ contains the preimage of an open $U \subseteq X_{i_0}$ for some index $i_0$, so it suffices to show that there is morphism $\Spa\, L' \to \tilde{X}$ (for $L'/L$ finite) whose image after projection to $X_{i_0}$ lies in $U$. To produce the desired map $\Spa\, L' \to \tilde{X}$, choose a classical point $x \in U(L')$ for $L'$ a finite extension of $L$ and then lift $x$ compatibly to an $L'$-valued points of each $X_j$, using that each $X_j \to X_{i_0}$ is strongly \'etale (by \cref{lem.composition-strongly-etale}) and that the residue field of $L'$ is algebraically closed since $L$ is strict. 

\end{proof}

\subsubsection{} Note that, for any object $\tilde{X}$ of $X_\proet$, we can restrict the completed structure sheaf $\hat{\mc{O}}$ on $X_\proet$ to a sheaf $\hat{\mc{O}|}_{\tilde{X}}$ on $|\tilde{X}|$ (using \cite[Lemma 3.10-(iii)]{scholze:p-adic-ht}). For any point $\tilde{x} \in |\tilde{X}|$, there is an associated valuation $v_{\tilde{x}}$ on the stalk $\hat{\mc{O}|}_{\tilde{X}}$ at $\tilde{x}$ induced by \cite[Lemma 4.2-(iv)]{scholze:p-adic-ht}. The following conjecture says that we expect potentially unramified towers over smooth rigid analytic varieties to have associated adic spaces with very nice differential properties analogous to smooth rigid analytic varieties: 

\begin{conjecture}\label{conj.pot-unram-prop}
Suppose $L$ is a $p$-adic field, $X/L$ is a smooth rigid analytic variety and $\tilde{X}/X$ is a potentially unramified tower. 
\begin{enumerate}
    \item  The triple $\hat{\tilde{X}}:=(|\tilde{X}|, \hat{\mc{O}}|_{|\tilde{X}|}, (v_{\tilde{x}})_{\tilde{x} \in |\tilde{X}|})$ is an adic space and $\left(\hat{\tilde{X}}\right)^\diamond=\tilde{X}^\diamond$.  
    \item The natural map $\mathcal{O}\mbb{B}^+_\dR|_{|\tilde{X}|} \rightarrow \hat{\mc{O}}|_{|\tilde{X}|}$
    is an isomorphism.
    \item Writing $f: \hat{\tilde{X}} \rightarrow X$ for the induced map of adic spaces, the derivation $d: \mc{O}_{\hat{\tilde{X}}} \rightarrow f^*\Omega_{X/L}$ induced by (2) and the usual derivation $\mc{O}\mbb{B}^+_\dR \rightarrow \mc{O}\mbb{B}^+_\dR \otimes_{\mc{O}} \Omega_{X/L}$ restricts, on any affinoid $\Spa(R,R^+)$ of $\hat{\tilde{X}}$, to the universal continuous derivation over $L$ from $R$ to an $L$-Banach module. 
    \item For $G/L$ a connected linear algebraic group and $[\mu]$ a conjugacy class of cocharacters of $G_{\overline{L}}$ defined over $L$, the Bialynicki-Birula map $\Gr_{[\mu]} \rightarrow \Fl_{[\mu^{-1}]}^\diamond$  induces a bijection between $\Gr_{[\mu]}(\tilde{X}^\diamond)$ and the subset of 
    \[ \Fl_{[\mu^{-1}]}(\tilde{X}^\diamond)=\Fl_{[\mu^{-1}]}(\hat{\tilde{X}})\] consisting of filtrations on the trivial $G$-torsor satisfying Griffiths transversality for the trivial connection. 
\end{enumerate}
\end{conjecture}

Note that \cref{conj.pot-unram-prop}-(4) is an analog of \cite[Theorem 5.4-(4)]{howe-klevdal:ap-pahsII} for rigid analytic varieties; more generally, one also expects an analog of \cite[Theorem 3.7]{howe-klevdal:ap-pahsII}. 

\subsection{Admissible pairs of (potential) good reduction}
\newcommand{\pgr}{\mathrm{pgr}}

\begin{definition}\label{defn:good-reduction-adm-pair}
Let $L$ be a $p$-adic field and let $X/\Spa\+ L$ be a smooth rigid analytic variety. An admissible pair $A$ on $X^\diamond$ has \emph{(potential) good reduction} if there is a cover $\{\tilde{U}_i \rightarrow X\}$ of $X$ by (potentially) unramified objects in $X_\proet$ such that $A|_{\tilde{U}_i^\diamond}$ is neutral. We write $\AdmPair^\mr{gr}$ (resp. $\AdmPair^{\pgr}(X^\diamond)$) for the full sub-category of $\AdmPair(X^\diamond)$ consisting of those admissible pairs of good reduction (resp. potential good reduction). 
\end{definition}

\begin{proposition}
$\AdmPair^{\mr{gr}}(X^\diamond)$ and $\AdmPair^{\pgr}(X^\diamond)$ are both abelian subcategories of $\AdmPair(X^\diamond)$ stable under tensor products and subquotient.
\end{proposition}
\begin{proof}
This follows from \cref{lemma.pur-product}. 
\end{proof}

\begin{example}\label{example.p-adic-field-gr-and-pgr}
Suppose $X = \Spd\+ L$ for $L$ a $p$-adic field. Let $C = \hat{\ol{L}}$ be the completion of an algebraic closure of $L$ and $\mf{G}_L = \Aut_{\mr{cont}}(C/L) \cong \mr{Gal}(\ol{L}/L)$. By descent, the category $\AdmPair(L)$ is equivalent to the category $\AdmPair(C)^{\mf{G}_L}$ of objects with a semilinear action of $\mf{G}_L$. The category $\AdmPair^{\mr{gr}}(L)$ (resp.\ $\AdmPair^{\mr{pgr}}(L)$) of admissible pairs with good reduction (resp.\ potential good reduction) is then equivalent to the subcategory of $\AdmPair(C)^{\mf{G}_L}$ of objects where the semilinear action on the associated isocrystal is unramified (resp.\ potentially unramified). 

The category $\AdmPair^{\mr{gr}}(L)$ has an explicit classical description. Let $L_0$ be the maximal absolutely unramified subfield of $L$, and let $\varphi\text{-}\mr{Mod}_{L/L_0}$ be the category of filtered $\varphi$-modules over $L$ \cite[\S 10.5]{fargues-fontaine:courbes}. Associated to each filtered $\varphi$-module $(W, \Fil^\bullet)$ is the pair $A(W, \Fil^\bullet) := (\mc{E}(W), \Fil^0 (W_L \otimes_{L} B_\dR))$ where $\mc{E}(W)$ in $\Vect(\FF_C)$ is the $\mf{G}_L$-equivariant HN-split vector bundle associated to the $\varphi$-module $W$, and $W_L \otimes_L B_\dR$ has the tensor product filtration for the usual filtration on $B_\dR = \mbb{B}_\dR(C)$. By the classification of vector bundles on the Fargues--Fontaine curve, $A(W, \Fil^\bullet)$ is an admissible pair if and only if $(W, \Fil^\bullet)$ is a weakly admissible filtered $\varphi$-module (see \cite[Proposition 10.5.6]{fargues-fontaine:courbes}). The association $(W, \Fil^\bullet) \mapsto A(W, \Fil^\bullet)$ thus is a fully faithful functor from weakly admissible filtered $\varphi$-modules over $L$ to $\AdmPair^{\mr{gr}}(L)$. To see that it is essentially surjective, suppose $(\mc{E}, \mc{L}_\et) \in \AdmPair(C)^{\mf{G}_L}$ has unramified $\mf{G}_L$-action. Then $\mc{E} = \mc{E}(W)$ for a $\varphi$-module $W$ over $L_0$ \cite[Proposition 10.2.2]{fargues-fontaine:courbes}, and by \cite[Proposition 10.4.3]{fargues-fontaine:courbes} the $\mf{G}_L$-lattice $\mc{L}_\et$ arises as $\Fil^0(W_L \otimes_L B_\dR)$ for a unique filtration $\Fil^\bullet$ on $W_L$. 
\end{example}

\subsubsection{} Let $\mc{G}$ be a $G$-admissible pair on $X^\diamond$. We say $\mc{G}$ has (potential) good reduction if $\mc{G}$ factors through $\AdmPair^\gr(X^\diamond)$ (resp. $\AdmPair^{\pgr}(X^\diamond)$). 

Note that, given a locally profinite group $H$ and an $H$-torsor $\tilde{X}/X^\diamond$, we obtain a canonical element of $X_\proet$, $\tilde{X}^\proet:=(\tilde{X}/K)_{K \leq H \textrm{ compact open}}$ by using the equivalence of the \'{e}tale site of $X$ and $X^\diamond$.  

\begin{proposition}\label{prop.torsor-unramified}
    A $G$-admissible pair $\mc{G}$ on $X^\diamond$ has (potential) good reduction if and only if, for each $[b] \in B(G)$, $(\tilde{X}^\diamond_{\mc{G},b})^\proet / X_{\mc{G},[b]}$ is (potentially) unramified. 
\end{proposition}
\begin{proof}
We treat the case of good reduction. The question is local on $X$ in the analytic topology, so assume that $X = X_{\mc{G}, [b]}$ for a single $[b] \in B(G)$, $X$ is affinoid, and we have a single unramified cover $\tilde{U} \to X$ with $\mc{G}|_{\tilde{U}^\diamond}$ neutral. Moreover, since $X$ is qcqs, we can assume that $\tilde{U}$ has a presentation $(U_i)_{i \in I}$ with each $U_i$ qcqs. Via the equivalence $X^\diamond_\et \cong X_\et$, the tower of \'etale covers $(\tilde{X}_{\mc{G},b}/K)_{K \leq G_b(\QQ_p)}$ (as $K$ ranges over compact open subgroups of $G_b(\QQ_p)$) arises from a unique tower $(X_K)_{K \leq G_b(\QQ_p)}$ in $X_\proet$. We need to show that this tower is unramified, i.e.\ that each $X_K$ is strongly \'etale over $X$. 

Fix $K \leq G_b(\QQ_p)$ a compact open subgroup. Since $\mc{G}|_{\tilde{U}^\diamond}$ is neutral, there is a map $\tilde{U}^\diamond \to \tilde{X}_{\mc{G},b}^\diamond$. Since $|U|$ is qc, the image of the composition 
    \[ |\tilde{U}| = |\tilde{U}^\diamond| \to |\tilde{X}_{\mc{G},b}^\diamond| \to |\tilde{X}_{\mc{G},b}^\diamond/K| = |X_K| \]
is contained in $|V|$ for $V \subseteq X_K$ a qcqs open. From \cite[Proposition 11.23-(ii), Lemma 15.6]{scholze:etale-cohomology-of-diamonds} we have equalities 
    \begin{align*}
         \Hom_{\tilde{U}^\diamond_\et}(\tilde{U}^\diamond, \tilde{U}^\diamond \times_{X^\diamond} V^\diamond) = \varinjlim_{j} \Hom_{U_{j, \et}}(U_j, U_j \times_{{X}} V),
    \end{align*}
and so the map $\tilde{U}^\diamond \to V^\diamond$ induces a splitting of $U_j \times_X V \to U_j$ for large enough $j$; let $s \colon U_j \to U_j \times_X V \to V$ be the composition of the splitting with projection to $V$. It follows from \cref{lem.composition-strongly-etale} that $X_K \to X$ is strongly \'etale at every point of $s(U_j)$. Replacing the map $\tilde{U}^\diamond \to \tilde{X}^\diamond_{\mc{G},b}$ by a $G_b(\QQ_p)$-translate, we can arrange for any point of $|X_K|$ to be in the image of $s$: indeed this follows since $|U_j| \to |X|$ is surjective and $G_b(\QQ_p)$ acts transitively on fibers of $\tilde{X}^\diamond_{\mc{G},b} \to X^\diamond$. We conclude that $X_K \to X$ is strongly \'etale. 

\end{proof}

\subsection{Variations of $p$-adic Hodge structure} Let $L$ be a $p$-adic field and let $X/\Spa\+ L$ be a smooth rigid analytic variety. In \cite[Definition 10.1]{howe-klevdal:ap-pahsII}, we defined a variation of $p$-adic Hodge structure on $X$. We recall this definition now, and define an equivalent full subcategory of $p$-adic Hodge structures on $X^\diamond$. Recall from \cite[\S3.3]{howe-klevdal:ap-pahsII} that for any de Rham local system $V$ on $X_\proet$, the results of \cite{scholze:p-adic-ht} produce a second $\mathbb{B}^+_\dR$ lattice $\mathbb{M}_0 \subseteq \mbb{L} \otimes_{\mbb{B}^+_\dR}\mbb{B}_\dR$ deforming the associated filtered vector bundle with integrable connection satisfying Griffiths transversality on $X_\et$\footnote{The lattice $\mathbb{M}_0$ (defined in the proof of \cite[Theorem 7.6-(ii)]{scholze:p-adic-ht}) is given by $(D_\dR(V) \otimes_\mc{O} \mc{O}\mbb{B}_\dR^+)^{\nabla = 0}$, where $D_\dR$ is the functor of \cite{liu-zhu:riemann-hilbert}.}. Recall furthermore from  \cite[\S3.3]{howe-klevdal:ap-pahsII} that the categories  of $\mbb{Q}_p$, $\mbb{B}^+_\dR$, and $\mbb{B}_\dR$-local systems on $X_\proet$ and $X^\diamond_v$ are naturally equivalent and thus it makes sense to discuss a de Rham local system on $X^\diamond$ and the associated lattice $\mathbb{M}_0$ also in this context. 

\begin{definition}\label{def.vpahs}
Let $L/\mathbb{Q}_p$ a $p$-adic field and let $X/L$ be a smooth rigid analytic variety.
\begin{enumerate}
\item A variation of $p$-adic Hodge structure on $X$ is a $\mbb{Q}$-graded de Rham $\mbb{Q}_p$-local system $V= \oplus_{\lambda} V_\lambda$ on $X_\proet$ such that, for any geometric point $x: \Spa(C,C^+) \rightarrow X$, $\oplus_{\lambda} (x^*V_{\lambda}, x^*\mbb{M}_0(V_\lambda))$
is a $p$-adic Hodge structure over $C$. We write $\VHS(X)$ for the category of variations of $p$-adic Hodge structure on $X$.
\item A variation of $p$-adic Hodge structure on $X^\diamond$ is a $p$-adic Hodge structure $\bigoplus (V_\lambda, \mc{L}_\lambda)$ on $X^\diamond$ such that, for each $\lambda \in \mbb{Q}$, $V_{\lambda}$ is a de Rham local system and $\mc{L}_{\lambda}=\mbb{M}_0(V_{\lambda})$. We write $\VHS(X^\diamond)$ for the category of variations of $p$-adic Hodge structure on $X^\diamond$. 
\end{enumerate}
\end{definition}

\cref{def.vpahs}-(1)/\cite[Definition 10.1]{howe-klevdal:ap-pahsII} was made using restriction to geometric points since in \cite{howe-klevdal:ap-pahsII} we defined only \emph{neutral} $p$-adic Hodge structures on $X^\diamond$. It is equivalent to require that $(V_\lambda, \mbb{M}_0(V_\lambda))$ be a $p$-adic Hodge structure, and this is essentially the content of the following lemma. 

\begin{lemma}
Let $L/\mathbb{Q}_p$ be a $p$-adic field and let $X/L$ be a smooth rigid analytic variety. The functor
\[ \oplus V_\lambda \mapsto  \oplus_{\lambda} (V_{\lambda}, \mbb{M}_0(V_\lambda)) \]
 is a tensor equivalence between $\VHS(X)$ and $\VHS(X^\diamond)$. 
\end{lemma}
\begin{proof}
    This is immediate, since the property of a $\mbb{Q}$-graded $\mbb{B}^+_\dR$-latticed local system being a $p$-adic Hodge structure can be checked at geometric points.
\end{proof}

\begin{proposition}\label{prop.basic-pgr-to-vhs}
Let $L/\mathbb{Q}_p$ be a $p$-adic field and let $X/L$ be a smooth rigid analytic variety. The functor $\omega_\et$ from $\AdmPair^\pgr(X^\diamond)$ to $\ul{\mathbb{Q}_p}$-local systems on $X^\diamond$ factors through the subcategory of de Rham local systems, and there is a canonical isomorphism $\mbb{M}_0 \circ \omega_\et = \omega_{\mc{L}_\dR}$.

In particular, the equivalence $\omega_{\et}^{\mc{L}}: \AdmPair^{\mr{basic}}(X^\diamond) \rightarrow \HS(X^\diamond)$ of \cref{theorem.ap-hs-properties} restricts to a functor
\begin{equation}\label{eq.basic-pgr-to-vhs} \AdmPair^{\mr{basic},\pgr}(X^\diamond) \rightarrow \VHS(X^\diamond). \end{equation}
\end{proposition}
\begin{proof}
The ``in particular" is an immediate consequence of the first statement.

For the first statement, we first need to show that the restriction of $\omega_\et$ to $\AdmPair^{\pgr}$ factors through de Rham local systems. This follows by the pointwise criterion of Liu-Zhu \cite{liu-zhu:riemann-hilbert}: if $A=(\mc{E}, \mathcal{L}_\et) \in \AdmPair^{\pgr}(X^\diamond)$ then to check that $\omega_{\et}(A)$ is de Rham, we can first base change to $\breve{L}$, then, for any classical point, it suffices to check after a finite extension; if the extension is large enough, then this becomes just the usual crystalline comparison as in \cref{example.p-adic-field-gr-and-pgr}. To see that the $\mathcal{L}_\dR = \mathbb{M}_0 \circ \omega_\et$, it also suffices to check this at classical points where it is clear from the construction of the usual crystalline comparison: indeed, if we pass to a trivializing cover $\tilde{X}_{A,b}$ for $A$, which is potentially unramified by \cref{prop.torsor-unramified} then, since $\Gr_G/\Spd\+ \mathbb{Q}_p$ is separated and the classical points of $\tilde{X}_b$ are dense by \cref{lemma.pur-density}, it suffices to check that $\mathbb{M}_0 \circ \omega_\et=\omega_{\mathcal{L}_\dR}(A)=\mc{E}\boxtimes \mathbb{B}^+_\dR$ at classical points (the classifying map $\tilde{X}_{A,b} \rightarrow \Gr_{\GL_n}$ parameterizing the lattice $\mathbb{M}_0 \circ \omega_\et \subseteq \mc{E}\boxtimes \mathbb{B}_\dR=\mc{E}_b \boxtimes \mathbb{B}_\dR$  must be the trivial one since they agree on a dense set). 
\end{proof}

\begin{conjecture}\label{conj.pot-good-red}
    The functor \cref{eq.basic-pgr-to-vhs} is essentially surjective, and thus, by \cref{prop.basic-pgr-to-vhs} defines an equivalence $\AdmPair^{\mr{basic},\pgr}(X^\diamond) \xrightarrow{\sim} \VHS(X^\diamond)$. 
\end{conjecture}

\begin{remark}
    When $X=\Spa\, L$, \cref{conj.pot-good-red} holds, and is induced by the usual potentially crystalline equivalence. Indeed, \cite[Proposition 10.2.2, 10.4.3]{fargues-fontaine:courbes} as in \cref{example.p-adic-field-gr-and-pgr} combine to factor the usual potentially crystalline equivalence through an equivalence between filtered $\varphi$-modules over $L_0$ (the maximal absolutely unramified subfield of $L$) equipped with a semilinear action of $\mf{G}_L = \mr{Gal}(\ol{L}/L)$ whose restriction to inertia is finite and admissible pairs with potential good reduction. That variations of $p$-adic Hodge structure are potentially crystalline Galois representations follows from de Rham = potentially semistable and the fact that an isoclinic isocrystal cannot be equipped with a monodromy operator.  
\end{remark}

\begin{remark}\label{remark.sen-theory}
    By construction, a neutralizing torsors as in \cref{prop.torsor-unramified} for the basic $G$-admissible pair associated to a $G$-variation of $p$-adic Hodge structure is a de Rham torsor whose associated filtered torsor with integrable connection has trivial filtration. It follows (e.g. from \cite[Theorem A]{Howe.GeometricSenAndKodairaSpencer}, though this is overkill) that the geometric Sen morphism of this torsor is identically zero. This strongly suggests that the torsor is potentially unramified, as would be implied by \cref{conj.pot-good-red} in light of \cref{prop.torsor-unramified}. Indeed, after extending scalars to a complete algebraically closed extension of $\mathbb{Q}_p$, it could be reasonable to conjecture that vanishing of geometric Sen is equivalent to being potentially unramified, just as vanishing of arithmetic Sen is equivalent to  potentially unramifed for Galois representations. 
\end{remark}

\subsection{The Hodge-Tate period map and transcendence}\label{ss.ht-period-transcendence}
We now make a conjecture generalizing one of the transcendence results of \cite{howe-klevdal:ap-pahsII}. 

\begin{conjecture}\label{conj.transcendence} Let $L$ be a strict $p$-adic field, let $C={\overline{L}}^\wedge$, let $X/L$ be a smooth rigid analytic variety, and let $\tilde{X}/X$ be a potentially unramified object in $X_\proet$ with $\tilde{X}^\diamond$ connected. Fix a point $x \in X(C)$. 

For $A$ a neutral admissible pair on $\tilde{X}^\diamond$ with motivic Galois group $G=\MG(A, x)$, and for $\mc{G}$ the associated $G$-admissible pair on $\tilde{X}^\diamond$:
\begin{enumerate}
    \item For some finite extension $L'/L$ and any connected component $Y$ of $\tilde{X}_{L'}^\diamond$, $\mc{G}|_{Y}$ has type $[\mu]$ for some conjugacy class of cocharacters $[\mu]$. 
    \item For $Y$ as in (1), $\pi_\HT: \tilde{Y}_{\mc{G},\et} \rightarrow \Fl_{[\mu]}^\diamond$ the associated Hodge-Tate period map, and $\tilde{Y}^\circ$ any connected component of $\tilde{Y}_{\mc{G}, \et, C}$, $\pi_\HT(\tilde{Y}^\circ)$ is rigid analytic Zariski dense in any connected open $U \subseteq \Fl_{[\mu]}^\diamond$ that it intersects.
\end{enumerate}

\end{conjecture}

The case of \cref{conj.transcendence} where $\tilde{X}=X$ is \cite[Theorem B]{howe-klevdal:ap-pahsII}. More, generally, if $X$ is connected and $A \in \AdmPair^\pgr(X)$, then $X=X_{A, [b]}$ for some $b$, and we can restrict $A$ to $\tilde{X}^\diamond$ for a connected component $\tilde{X}$ of $\tilde{X}_{A,b}$ to obtain an object where \cref{conj.transcendence} applies. In particular, if \cref{conj.pot-good-red} holds, this can be applied to any variation of $p$-adic Hodge structure.

\subsubsection{}

The proof of \cite[Theorem B]{howe-klevdal:ap-pahsII}  proceeds by establishing the existence of a Hodge generic classical point. This existence depends on properties of the Hodge period map. We now discuss which aspects of the proof go through still in the more general context of \cref{conj.transcendence} and how this relates to \cref{conj.pot-unram-prop} on the geometric and differential properties of potentially unramified objects. 

Suppose we are in the setting of \cref{conj.transcendence}, and that (1) holds. We then have an \'{e}tale lattice period map $\pi_{\mc{L}_\et}: Y \rightarrow \Gr_{[\mu]}$ and the associated Hodge period map $\pi_{\Hdg}: Y \rightarrow \Fl_{[\mu^{-1}]}.$  Using \cref{lemma.pur-density}, we find that one of the important steps in the proof of \cite[Theorem B]{howe-klevdal:ap-pahsII} still holds in this setting: as in \cite[Lemma 7.16]{howe-klevdal:ap-pahsII}, the Hodge loci, which, by \cite[Lemma 7.10]{howe-klevdal:ap-pahsII} are closed subdiamonds for any neutral admissible pair on any diamond, are in this case pulled back along $\pi_\Hdg$ from Zariski closed subsets of $\Fl_{[\mu^{-1}]}$ (as in loc. cit., one argues the natural Hodge tensor conditions formulated using the lattice and filtration are equivalent at rigid analytic points, then extends by the density of rigid analytic points in $Y^\diamond$). 

The key missing ingredient in the proof of the existence of a Hodge-generic classical point is then the analog of weakly Shilov points and \cite[Lemma 2.1.4]{bhatt-hansen:six-functors-Zariski-constructible}, which allowed us in the rigid analytic setting to show the Hodge generic locus contained many rank one points. If \cref{conj.pot-unram-prop} holds, then we can replace $\pi_\Hdg$ with an actual map of adic spaces $\pi_{\Hdg}: \hat{Y} \rightarrow \Fl_{[\mu^{-1}]}$, and in the spirit of \cref{conj.pot-unram-prop}, we expect that analogous properties should hold also for Zariski closed subsets of $\hat{Y}$.

\subsection{Differential aspects of the Hodge period map}\label{ss.inscribed-period-maps}
In the basic case, we proved in \cite[Theorem 9.4]{howe-klevdal:ap-pahsII} a more refined transcendence result than \cite[Theorem B]{howe-klevdal:ap-pahsII}. For the non-minuscule case of \cite[Theorem 9.4]{howe-klevdal:ap-pahsII}, the characterization of maps from rigid analytic varieties to the $\mbb{B}^+_\dR$-affine Grassmannian of \cite[Theorem 5.4-(4)]{howe-klevdal:ap-pahsII} also plays an important role. In the more general setting of an admissible pair with good reduction, it is thus natural to study differential properties of the Hodge and \'{e}tale lattice period map on $\tilde{X}_b^\diamond$. 

A good differential theory is not typically accessible directly at the level of diamonds. However, if $\mc{G}$ has potential good reduction, then $\tilde{X}_b/X$ is potentially unramified, and \cref{conj.pot-unram-prop} would imply that there is an associated adic space $\hat{\tilde{X}}_b$ and that we can view the Hodge period map as a map of adic spaces 
\[ \pi_{\Hdg}: \hat{\tilde{X}}_b \rightarrow \Fl_{[\mu^{-1}]}. \]
Furthermore, \cref{conj.pot-unram-prop} would imply that, viewed as a map of adic spaces, $d\pi_\Hdg$ is naturally identified with the pullback to $\hat{\tilde{X}}_b$ of the Kodaira-Spencer map for the associated filtered $G$-bundle with connection on $X$.

\subsubsection{}The theory of \cite{Howe.InscriptionTwistorsAndPAdicPeriods} gives an alternative notion of differential structures for which this kind of statement can be made precise unconditionally. Explicitly, associated to $X$ there is an inscribed $v$-sheaf $X^\lfid$ as in \cite[\S4.4]{Howe.InscriptionTwistorsAndPAdicPeriods} with a Hodge period map $\pi_{\Hdg}: X^\lfid \rightarrow \Fl_{[\mu^{-1}]}^\lfid / G(\overline{\mc{O}})$ whose derivative is the Kodaira--Spencer map (obtained by applying \cite[\S 7.1.3-7.1.4]{Howe.InscriptionTwistorsAndPAdicPeriods} to the $G$-bundle with integrable connection $D_\dR \circ \mc{G}$ for $D_\dR$ the relative de Rham functor of \cite[Theorem 1.5]{liu-zhu:riemann-hilbert}). As in \cite[\S 7.2]{Howe.InscriptionTwistorsAndPAdicPeriods}, there is a lattice lift 
\[ \pi^+_{\Hdg}: X^\lfid \rightarrow \Gr_{[\mu]} /G(\overline{\mathbb{B}^+_\dR}) \]
where $\Gr_{[\mu]}$ is a natural inscribed upgrade of the $\mathbb{ B}^+_\dR$-affine Grassmannian; the derivative of $\pi^+_\Hdg$ is the canonical lift of the Kodaira-Spencer morphism (it is also the canonical twisted Higgs field / geometric Sen morphism \cite[\S1.2.6]{Howe.GeometricSenAndKodairaSpencer}). 

The pullback to  $X^\lfid$ along $\pi_{\Hdg}^+$ of the canonical $G(\overline{\mathbb{B}^+_\dR})$-torsor is the (flat) trivializing torsor of $\omega_{\mathcal{L}_\dR} \circ \mc{G}$. In particular, letting $\hat{\tilde{X}}_b^\lfid$ be the pullback of $\tilde{X}_b^\diamond / X^\diamond$ along the natural map $X^\lfid \rightarrow X^\diamond$, it follows from the constructions that $\hat{\tilde{X}}_b^\lfid$ is a reduction of structure group to $G_b(\mathbb{Q}_p^\diamond)$ for the trivializing torsor of $\omega_{\mathcal{L}_\dR} \circ \mc{G}$ and thus we obtain a natural period map 
\[ \pi_{\Hdg}^+: \hat{\tilde{X}}_b^\lfid \rightarrow \Gr_{[\mu]} \]
and, by composition with the Bialynicki-Birula map
\[ \pi_{\Hdg}: \hat{\tilde{X}}_b^\lfid \rightarrow \Fl^{\lfid}_{[\mu^{-1}]}. \]
Writing $q: \hat{\tilde{X}}_b^\lfid \rightarrow X^\lfid$, $dq$ gives a canonical identification $q^*  T_X=T_{\hat{\tilde{X}}_b^\lfid}$, and under this identification $d\pi_\Hdg$ is the Kodaira-Spencer map. 

We view this inscribed version of the construction as some further evidence that \cref{conj.pot-unram-prop} is reasonable and that the geometric transcendence results of \cite{howe-klevdal:ap-pahsII} may extend to more general admissible pairs of potential good reduction (and, thus, via \cref{conj.pot-good-red}, to variations of $p$-adic Hodge structure).  

\begin{remark}
    The key missing ingredient in the inscribed approach is that we do not know whether functions on $\hat{\tilde{X}}_b$ induce canonical maps $\hat{\tilde{X}}_b^\lfid \rightarrow \mathbb{A}^{1,\lfid}$, thus we do not yet know that such functions can be differentiated as conjectured in \cref{conj.pot-unram-prop}. On the other hand, one advantage of the inscribed approach is that we also obtain a differentiable version of $\pi_{\Hdg}^+$ even though, when $[\mu]$ is not miniscule, there does not exist a rigid analytic version of $\Gr_{[\mu]}$. 
\end{remark}

\begin{remark}
    If we treat $\tilde{X}_b/X$ as a $G_b(\mathbb{Q}_p)$-torsor, then the construction of \cite[\S11]{Howe.InscriptionTwistorsAndPAdicPeriods} produces an inscribed $v$-sheaf $\tilde{X}_b^\lfid$ that is a $G(\mathbb{Q}_p^\lfid)$ quasi-torsor over $X^\lfid$. In terms of our above construction, it can be described as the push-out
    \[ \tilde{X}_b^\lfid =  \hat{\tilde{X}}_b^\lfid \times^{G_b(\mathbb{Q}_p^\diamond)} G_b(\mathbb{Q}_p^\lfid).\]
    In particular, the tangent bundle of $\tilde{X}_b^\lfid$ is the direct sum of $T_X$ and the constant bundle $\Lie\+ G_b(\mathbb{Q}_p) \otimes_{\mbb{Q}_p} \mathbb{Q}_p^\lfid$ in the vertical direction. We emphasize that the existence of such a reduction is atypical for $p$-adic Lie group torsors and the construction of \cite[\S 11]{Howe.InscriptionTwistorsAndPAdicPeriods}; the existence here reflects the fact that $\tilde{X}_b$ is potentially unramified. In particular, the splitting of the tangent bundle of $\tilde{X}_b^\lfid$ can be deduced from the fact that the Kodaira-Spencer / geometric Sen morphism for $\tilde{X}_b/X$ (which is \emph{not} the same as the Kodaira-Spencer morphism / geometric Sen morphism for $\tilde{X}/X$) is identically zero (cf. \cref{remark.sen-theory}).   
\end{remark}

\section{Special points}\label{s.special-points}
In this section, we study special points of the moduli spaces $\Gr_{[\mu]}^{\badm}$ and $\mc{M}_{b, [\mu]}$. In analogy with special points on global Shimura varieties, we define the special points to be the locus of points at which the universal admissible pair has motivic Galois group a torus. For global Shimura varieties special points are known to satisfy some remarkable properties, most notably the \emph{Andr\'{e}-Oort conjecture} (proven by \cite{pila-shankar-tsimerman:andre-oort}) that the Zariski closure of a set of special points is a special subvariety. A more classical result is that special points are dense in the analytic topology. 

We investigate the extent to which $p$-adic analogues of these results hold for $\Gr_{[\mu]}^\badm$. Though some of our results hold for general $[\mu]$ (e.g.\ the existence of special points when $b$ is basic in \cref{lemma.existence-special-point}), we primarily restrict to $[\mu]$ minuscule in order to avoid subtleties in talking about a `rigid analytic Zariski topology' of $\Gr_{[\mu]}^\badm$ when $[\mu]$ is not minuscule (cf. \cite[Remark 9.2]{howe-klevdal:ap-pahsII}). When $[\mu]$ is minuscule, $\Gr_{[\mu]}^\badm = (\Fl_{[\mu^{-1}]}^\badm)^\diamond$, and we show that special points are rigid Zariski dense in $\Fl_{[\mu^{-1}]}^\badm$ if and only if $b$ is basic (\cref{theorem.special-points-dense}). Finally, in \cref{ss.p-adic-andre-oort} we construct a simple example to show that the na\"ive analogue of the Andr\'e-Oort conjecture fails in this setting (\cref{example.andre-oort-counterexample}).

\subsection{The structure of the special points}\label{ss.special-points-structure}
Let $G/\mathbb{Q}_p$ be a connected reductive group and let $[\mu]$ be a conjugacy class of cocharacters of $G_{\overline{\mbb{Q}}_p}$ and $b$ a representative of a class $[b] \in B(G, [\mu^{-1}])$. We first recall the definition of special points of $\Gr_{[\mu]}^{b-\adm}$ and $\mc{M}_{b, [\mu]}$. 

\begin{definition}[Special points]\label{defn.special-points}
For $C/\breve{\mbb{Q}}_p([\mu])$ an algebraically closed non-archimedean extension, a \emph{special point} of $\Gr_{[\mu]}^{b-\adm}(C)$ (respectively $\mc{M}_{b,[\mu]}(C)$) is a point $x$ such that the motivic Galois group of the restriction of the universal $G$-admissible pair to $\Spd\+ C$ is a torus.
\end{definition}

Equivalently, the special points are the $C$-points in the image of a functoriality map as in \cite[\S 7.5]{howe-klevdal:ap-pahsII}
\[ \Gr_{\mu_T}^{b_T-\adm} \rightarrow \Gr_{[\mu], \Qpbreve(\mu_T)}^{b-\adm} \qquad \left(\text{respectively } \mc{M}_{b_T, \mu_T} \to \mc{M}_{b, [\mu], \Qpbreve(\mu_T)}\right) \]
for $T \leq G$ a torus, $\mu_T$ a cocharacter of $T_{\overline{\mbb{Q}}_p}$ lying in $[\mu]$, $b_T$ a representative of the unique class in $B(T,\mu_T^{-1})$ (by \cite[Proposition 7.10]{howe-klevdal:ap-pahsII}, requiring this compatibility is equivalent to the space on the left being non-empty), and a choice of $g \in G(\Qpbreve)$ such that $gb_T\sigma(g^{-1})=b$. Note that, in this case, $\Gr_{\mu_T}^{b_T-\adm}=\Gr_{\mu_T, \Qpbreve}=\Spd\+ \Qpbreve(\mu_T)$, so that the associated special points are defined already over $\Qpbreve(\mu_T)$.

\begin{proposition}\label{prop.special-points-loc}
The locus of special points in $\Gr_{[\mu]}^{b-\adm}$ is a finite union of $G_b(\mathbb{Q}_p)$-orbits, each defined over a finite extension of $\Qpbreve([\mu])$.  
\end{proposition}
\begin{proof}
    We first note that, in the characterization of special points via functoriality maps, we may always enlarge our torus $T$ to a maximal torus. Two maximal tori that are conjugate under $G(\mathbb{Q}_p)$ give rise to the same associated special points (for $\mc{G}$ the $G$-admissible pair associated to a special point, the different choices correspond to fixing different trivializations of the fiber functor $\omega_\et \circ \mc{G}$ to identify the motivic Galois group with a subgroup of $G$), thus it suffices to consider one $T$ in each $G(\mathbb{Q}_p)$-conjugacy class of maximal tori. 
    
    By \cite[Corollary 1 on p.147]{Serre.GaloisCohomology}, there are only finitely many such conjugacy classes of maximal tori in $G$. For a fixed maximal torus $T$, there are only finitely many choices of $\mu_T$ living in $[\mu]$. For a given $(T,\mu_T)$, the different choices of $b_T$ give rise to the same set of special points, and for a fixed choice the collection of all associated special points form an orbit for $G_b(\mathbb{Q}_p)$ (via the choice of $g$ that $\sigma$-conjugates $b_T$ to $b$). Since any point in this orbit is defined over $\Qpbreve(\mu_T)$ as above, we conclude. 

\end{proof}

\begin{corollary}The locus of special points in $\mc{M}_{b,[\mu]}$ is a finite union of $G_b(\mbb{Q}_p) \times G(\mbb{Q}_p)$-orbits. Moreover, the points in each orbit are defined over the perfectoid field $E_\infty$ for $E/\mbb{Q}_p([\mu])$ a finite extension and $E_\infty$ the completion of its maximal abelian extension. 
\end{corollary}
\begin{proof}
    The first claim is immediate since $\pi_{\mc{L}_\et}:\mc{M}_{b,[\mu]} \rightarrow \Gr_{[\mu]}^{b-\adm}$ is a $G(\mbb{Q}_p)$-torsor and the special points in $\mc{M}_{b,[\mu]}$ are precisely the pre-images of special points in $\Gr_{[\mu]}^{b-\adm}$. To see the claim about the field of definition, it suffices to treat the case where $G=T$ is a torus. In this case, by picking $b$ to lie in $T(\mbb{Q}_{p^n})$, we can define $\Gr_{[\mu]}^{b-\adm}$ over $\mbb{Q}_{p^n}(\mu)$, where it is isomorphic to $\Spd(\mbb{Q}_{p^n}(\mu))$. Then, since the structure group $T(\mbb{Q}_p)$ is abelian, the points of $\mc{M}_{b,[\mu]}$ are defined over (the completion of) the maximal abelian extension of $\mbb{Q}_{p^n}(\mu)$. 
\end{proof}

\begin{remark}
    The stabilizer in $G_b(\mbb{Q}_p)$ of a point $x \in \Gr_{[\mu]}^{b-\adm}(C)$ is naturally identified with the automorphism group $A$ of the associated $G$-admissible pair $\mc{G}$. Thus the orbit of $x$ is naturally identified with the locally profinite set $G_b(\mbb{Q}_p)/A$. Moreover, if we fix a trivialization of $\omega_\et \circ \mc{G}$, i.e. a point $x_\infty \in \mc{M}_{b, [\mu]}$ lying over $x\in\Gr_{[\mu]}^{b-\adm}(C)$, then the automorphism group $A$ is identified with the $\mathbb{Q}_p$-points of the centralizer of the motivic Galois group of $\mc{G}$, viewed as a subgroup of $G$ via this choice of trivialization. Thus the orbit of $x_\infty$ is identified with the locally profinite set $(G_b(\mbb{Q}_p) \times G(\mbb{Q}_p)) / A$, where $A$ acts by $(g_b,g)\cdot a= (g_b a, a^{-1}g)$. In particular, for $K \leq G(\mbb{Q}_p)$ a compact open subgroup, the corresponding special points in $\mc{M}_{b,[\mu]}/K$ are given by
    \[ (G_b(\mbb{Q}_p) \times G(\mbb{Q}_p)/K) / A. \]
\end{remark}

Note that it is not clear that special points exist at all. The following lemma, which is essentially \cite[Proposition 1.25]{rapoport-zink:period-spaces}, establishes this existence in the basic case. We give a proof for completeness.  

\begin{lemma}\label{lemma.existence-special-point}
    Let $G/\mathbb{Q}_p$ be a connected reductive group, let $[\mu]$ be a conjugacy class of cocharacters of $G_{\overline{\mbb{Q}}_p}$, and let $b \in G(\Qpbreve)$ represent the unique basic class $[b]\in B(G,[\mu^{-1}])$. Then, $\Gr_{[\mu]}^{b-\adm}$ contains a special point. 
\end{lemma}
\begin{proof}
Applying  \cite[Theorem 6.21]{platonov-rapinchuk}, we find there is an elliptic (i.e. anisotropic modulo center) maximal torus $T \leq G$. Fix a cocharacter $\mu_T$ of $T_{\overline{\mathbb{Q}}_p}$ mapping to $[\mu]$, and let $[b_T]$ be the unique element in $B(T,\mu_T^{-1})$. By \cite[Proposition 5.3]{kottwitz:isocrystals}, the image of $[b_T]$ in $B(G)$ is basic. Then by combining \cite[Remark 5.8]{kottwitz:isocrystals} and \cite[6.4]{kottwitz:isocrystals-II}, we find that the image of $[b_T]$ is the unique basic conjugacy class in $B(G, [\mu^{-1}])$, i.e. $[b]$. It follows from \cite[Proposition 3.1]{rapoport-viehmann} (though this is overkill for a torus!) that 
\[ \Gr_{\mu_T}^{b-\adm}=\Fl_{\mu_T^{-1}}^{b_T-\adm, \diamond} \neq \emptyset,\]
thus it must be equal to $\Fl_{\mu_T^{-1}}^\diamond=\Spd\+ L(\mu_T)$. In particular, if we choose a $g \in G(\Qpbreve)$ such that $gb_T\sigma(g^{-1})=b_G$ inducing a functoriality map as in \cite[\S 7.5]{howe-klevdal:ap-pahsII}, then we obtain a special point $\Spd\+ L(\mu_T) \rightarrow \Gr_{[\mu], L(\mu_T)}^{b-\adm}$.  
\end{proof}

\begin{remark}\label{remark.existence-of-special-points} A complete characterization of those $(G,b,[\mu])$ with $[\mu]$ minuscule such that $\Fl_{[\mu^{-1}]}^{b-\adm}$ contains a special point is described in \cite[Theorem 1 and the following two paragraphs]{KisinMadapusiShin.HondaTateTheoryForShimuraVarieties}: if there is a special point, then one can $\sigma$-conjugate so that $b$ arises from a basic element of a torus. Then, the slope morphism can be defined over $\mbb{Q}_p$ and $b$ is also a basic element for the centralizing Levi subgroup. Conversely, if $b$ arises from a basic element of a Levi subgroup then, applying  \cite[Corollary 1.1.15]{KisinMadapusiShin.HondaTateTheoryForShimuraVarieties}, one reduces the existence of a special point in $\Gr_{[\mu]}^{b-\adm}$ to the basic case treated in \cref{lemma.existence-special-point}. Note that this also characterizes the local Shimura varieties that contain special points. We thank Keerthi Madapusi for pointing out this result; it would be interesting to know if it extends to the non-minuscule case. 
\end{remark}

\subsection{Density of special points}

We now restrict to the case of minuscule $[\mu]$, where $\Gr_{[\mu]}=\Fl_{[\mu^{-1}]}^\diamond$, and thus the non-empty open $\Gr_{[\mu]}^{b-\adm}$ is the diamondification of a unique non-empty open rigid analytic subvariety $\Fl_{[\mu^{-1}]}^{b-\adm} \subseteq \Fl_{[\mu^{-1}]}$. 

\begin{theorem}\label{theorem.special-points-dense}
     Let $G/\mathbb{Q}_p$ be a connected reductive group, let $[\mu]$ be a minuscule conjugacy class of cocharacters of $G_{\overline{\mbb{Q}}_p}$, and let $b \in G(\Qpbreve)$ represent a class $[b]\in B(G,[\mu^{-1}])$. Then, for $L/\Qpbreve([\mu])$ any non-archimedean extension, the special points are rigid Zariski dense in $\Fl_{[\mu^{-1}], L}^{b-\adm}$ if and only if $[b]$ is basic.
\end{theorem}
\begin{proof}
Suppose $[b]$ is basic. Then, $G_b$ is an inner form of $G$. In particular, $\Lie\+ G_b(\mbb{Q}_p)$ spans $\Lie\+ G_{\Qpbreve}$. Thus, any orbit of $G_b(\mbb{Q}_p)$ on $\Fl_{[\mu^{-1}], L}$ is rigid Zariski dense in any open $U$ such that each connected component of $U$ intersects the orbit (e.g., because the values of a rigid analytic function restricted to the orbit contain enough information to compute its partial derivatives of all orders at any point in the orbit, and thus its power series expansion at any such point). Since $\Fl_{[\mu^{-1}],L}^{b-\adm}$ is connected by the results of \cite{gleason-lourenco:connectedness} (see \cite[Proposition 7.11]{howe-klevdal:ap-pahsII} for the statement with the same notation we use here), and contains a special point by \cref{lemma.existence-special-point}, we conclude. 

If $b$ is not basic, it suffices to show that each of the finitely many $G_b(\mbb{Q}_p)$-orbits of \cref{prop.special-points-loc} is contained in a strict special subvariety (since these are rigid Zariski closed by \cite[Lemma 7.22 and Lemma 2.3]{howe-klevdal:ap-pahsII}). For $(T, \mu_T)$ a pair consisting of a torus and a cocharacter giving rise to the orbit, we may assume that $b$ lies in $T(\Qpbreve)$. Then, as in \cref{remark.existence-of-special-points}, since $b$ is basic for $T$, the slope morphism $\nu_b$ is defined over $\mathbb{Q}_p$. Thus, so is its centralizer $M$, and $G_b$ is an inner form of $M$. In particular, writing $\mu_M$ for $\mu_T$ viewed as a cocharacter of $M$, the entire $G_b(\mathbb{Q}_p)$-orbit of the associated special point is contained in the special subvariety $\Fl_{[\mu_M^{-1}],C}^{b-\adm} \subseteq \Fl_{[\mu^{-1}],C}^{b-\adm}$. This containment is strict: since these are open subsets, it suffices to show that if the flag variety for $M$ has the same dimension as the flag variety for $G$, then $M=G$ (which, since we are outside the basic case, we have assumed does not hold!). But they have the same dimension if and only if, assuming $M_C$ and $P_{\mu^{-1},C}$ are a standard Levi and parabolic, the unipotent radical of $P_{\mu^{-1},C}$ is contained in $M_C$. Passing to the adjoint group and simple factors, we may assume $G$ is simple. Then, if $P_{\mu^{-1},C}$ is a proper parabolic subgroup, the corresponding factor of $M_C$ contains the highest root, and thus is equal to $G$. On the other hand, if $P_{\mu^{-1},C}=G$, then $\mu$ is central and so it follows that the $b$ is basic (since, e.g., in this case the Hodge-Tate period domain $\Fl_{[\mu]}$ has only one Newton stratum, but we know the basic stratum is non-empty); since $M_C$ is the centralizer of the slope morphism, it follows also in this case that $M=G$.  

\end{proof}

\begin{remark}\label{remark.special-subvarieties}
From the proof of \cref{theorem.special-points-dense}, we find that in the non-basic case each $G_b(\mathbb{Q}_p)$-orbit of special points lies in a strict basic special subvariety associated to a Levi subgroup of $G$ that is a $\mathbb{Q}_p$-form of the centralizer of the slope morphism $\nu_b$. 
\end{remark}

In this minuscule setting, the \'{e}tale covers $\mathcal{M}_{b,[\mu]}/K \rightarrow \Gr_{[\mu]}^{b-\adm}$ arise from the diamondification of rigid analytic \'{e}tale covers $\mathcal{M}_{b,[\mu],K} \rightarrow \Fl_{[\mu^{-1}]}^{b-\adm}$.  These are the local Shimura varieties of \cite{scholze:berkeley}, first conjectured to exist in \cite{rapoport-viehmann}. Since the special points in $\mc{M}_{b,[\mu],K}$ are the preimage of the special points in $\Fl_{[\mu^{-1}]}^{b-\adm}$, we obtain as an immediate consequence

\begin{corollary}\label{corollary.density-of-special-points}
 Let $G/\mathbb{Q}_p$ be a connected reductive group, let $[\mu]$ be a minuscule conjugacy class of cocharacters of $G_{\overline{\mbb{Q}}_p}$, and let $b \in G(\Qpbreve)$ represent a class $[b]\in B(G,[\mu^{-1}])$. The special points are rigid Zariski dense in the local Shimura variety $\mc{M}_{b,[\mu],K}$ if and only if $[b]$ is basic. 
\end{corollary}

\begin{example}\label{example.serre-tate}
    Let $G=\GL_2$, let $\mu=\mr{diag}(z^{-1},1)$, and let $b=\mr{diag}(p,1)$. This is the case associated to the Serre-Tate deformation space of an ordinary elliptic curve. The admissible locus is $\mbb{A}^1_{\Qpbreve} \subseteq \mbb{P}^1_{\Qpbreve}$, and,    
    for $K=\GL_2(\mbb{Z}_p)$, the associated local Shimura variety is given by the open rigid analytic unit disk of $|z-1|<1$, and the Hodge period map is $z \mapsto \log(z)$. The only conjugacy class of tori giving rise to special points it the split class containing the diagonal torus $T$, and it gives rise to a unique special point in the admissible locus, $0 \in \mbb{A}^1_{\Qpbreve}$. Its preimage in the local Shimura variety is thus the set of roots of unity, which is Zariski closed as the vanishing locus of the rigid analytic function $\log(z)$. 
\end{example}

\begin{remark}
    In \cite{howe-klevdal:ap-pahsII}, we considered a notion of rigid Zariski closure that can be applied to $\Gr_{[\mu],L}^{b-\adm}$ even when $[\mu]$ is not minuscule --- the rigid Zariski closure of a set $S \subseteq |\Gr_{[\mu],L}^{b-\adm}|$ is the intersection of all closed rigid analytic subdiamonds containing $S$. If $b$ is basic, $[\mu]$ is \emph{not} minuscule, and $L$ is a $p$-adic field then, by arguing with the $G_b(\mathbb{Q}_p)$-action as in \cref{theorem.special-points-dense}, we find again that the special points are rigid Zariski dense. However, now the reason is because they are not contained in \emph{any} rigid analytic subdiamond (if they were contained in a rigid analytic subdiamond $Z$, the composition of the inclusion of $Z$ with the Bialynicki-Birula map would fail Griffiths transversality, contradicting \cite[Theorem 5.4-(4)]{howe-klevdal:ap-pahsII}). The statement is no longer an if and only if: if $b$ is not basic, but $[\mu]$, viewed as a conjugacy class of cocharacters of the centralizer of the slope morphism, is non-minuscule, then the same argument shows that if there are any special points then they are rigid Zariski dense. We expect similar statements to hold for arbitrary $L$ --- more precisely, we expect that maps to the flag variety from rigid analytic varieties that factor through $\Gr_{[\mu],L}$ always satisfy Griffiths transversality, even though this will not induce a bijection as in \cite[Theorem 5.4-4]{howe-klevdal:ap-pahsII} if $L$ is not discretely valued (this expectation is supported by the differential perspective on the Bialynicki-Birula map of \cite[Corollary 6.3.2]{Howe.InscriptionTwistorsAndPAdicPeriods}). 
\end{remark}

\subsection{Failure of a $p$-adic Andr\'e-Oort conjecture}\label{ss.p-adic-andre-oort}

Let 
    \[ b = {\small \begin{pmatrix} & 1 & & \\ & & \ddots & \\ & & & 1 \\ p & & & \end{pmatrix} } \in \GL_n(\breve{\QQ}_p).\]
The class $[b]$ is the unique basic class in $B(\GL_n, [\mu^{-1}])$ for $\mu(t) = \mr{diag}(t^{-1}, 1, \ldots, 1)$. Since $\mu$ is minuscule, $\Gr_{[\mu]} = \Fl_{[\mu^{-1}]}^\diamond$. The flag variety is simply projective space $\mathbb{P}^{n-1}$, and in this case the $b$-admissible locus is equal to all of $\mbb{P}^{n-1}_{\Qpbreve}$. Indeed, the associated local Shimura variety at level $\GL_n(\mbb{Z}_p)$, $\mc{M}_{b,[\mu],\GL_n(\mbb{Z}_p)}$, can be identified with a disjoint union of copies of the height $n$ Lubin-Tate space (which is an open unit disk of dimension $n-1$ arising as the rigid generic fiber of the formal deformation space of the unique height $n$ one-dimensional formal group over $\overline{\mbb{F}}_p$) such that the finite level Hodge period map $\mc{M}_{b,[\mu],\GL_n(\mbb{Z}_p)} \rightarrow \mbb{P}^{n-1}_{\Qpbreve}$ restricts on each component to the Gross-Hopkins period map, which is surjective by \cite[Proposition 23.5]{gross-hopkins:equivariant}\footnote{This surjectivity was used in the first proof of the classification of vector bundles on the Fargues--Fontaine curve \cite{fargues-fontaine:courbes}. Fargues and Scholze \cite[II.2]{fargues-scholze:geometrization} gave a new proof of the classification that is independent of this computation, and thus one can instead use the classification to show the period map is surjective directly from the definition in terms of modifications of bundles without having to make the link between modifications and the theory of $p$-divisible groups. }.

The group $G_b(\mbb{Q}_p)$ can be identified very explicitly inside of $\GL_n(\breve{\QQ}_p)$: first, we consider the embedding 
\[ \iota: \mathbb{Q}_{p^n} \hookrightarrow M_n(\Qpbreve), a \mapsto \mr{diag}(a, \sigma^{n-1}(a), \sigma^{n-2}(a), \ldots, \sigma(a)). \]
Then we can construct a division algebra $D$ inside of $M_n(\Qpbreve)$ as the set of linear combinations 
\begin{equation} \label{eq.D-expression} \iota(a_0) + \iota(a_1) b + \iota(a_2) b^2 + \ldots \iota(a_{n-1}) b^{n-1},\; a_i \in \mbb{Q}_{p^n}.  \end{equation}
Then, it is elementary to check that $G_b(\mbb{Q}_p) = D^\times \subseteq \GL_n(\Qpbreve)$, i.e. it is the set of those combinations where not all coefficients $a_i$ are zero. 

In particular, since these matrices all lie in $\GL_n(\mbb{Q}_{p^n})$, the action of $D^\times$ is defined already on $\mbb{P}^{n-1}_{\QQ_{p^n}}$. We claim it is moreover transitive on $\mbb{P}^{n-1}(\QQ_{p^n})$. Indeed, the first column of the matrix in \cref{eq.D-expression} is given by
\[ \begin{bmatrix} a_0 \\ p\sigma^{n-1}(a_{n-1}) \\ \vdots \\ p \sigma(a_1) \end{bmatrix} \]
Since any vector in $\QQ_{p^n}^n$ can be hit by making a suitable choice of the $a_i$, it follows that any element of $\mbb{P}^{n-1}(\QQ_{p^n})$ is in the $D^\times$-orbit of $[1:0:\ldots:0].$ On the other hand, $\mbb{Q}_{p^n}^\times \subseteq D^\times$ is the stabilizer of the point $[1:0:\ldots:0]$. Since this is a torus, $[1:0:\ldots:0]$ is a special point, and we conclude every point of $\mbb{P}^{n-1}(\QQ_{p^n})$ is a special point (this is the unique $D^\times$-orbit of special points corresponding to the elliptic torus $\mathrm{Res}_{\mbb{Q}_{p^n} /\mbb{Q}_p} \mbb{G}_m \leq \GL_{n,\mbb{Q}_p}$ as in \cref{ss.special-points-structure}). 

As a consequence of these computations, we find:
\begin{proposition}\label{prop.rigid-zar-dense}
    Let $Z \subseteq \mbb{P}^{n-1}_{\mbb{Q}_{p^n}}$ be a connected rigid subvariety such that $Z(\mbb{Q}_{p^n})$ contains a smooth point. Then, for any $L/\breve{\mbb{Q}}_p$, special points are rigid Zariski dense in $Z_L \subseteq \mbb{P}^{n-1}_{L}=\mbb{P}_{L}^{n-1, b-\adm}$. 
\end{proposition}
\begin{proof}
    Let $z \in Z(\mbb{Q}_{p^n})$ be a smooth point. Then, $Z(\mbb{Q}_p^n)$ contains a ball of dimension $\dim Z$ around $x$. By the above discussion, this ball consists entirely of special points. On the other hand, such a ball is rigid Zariski dense in $Z_L$ for any $L$ --- indeed, the values of a rigid analytic function restricted to such a ball contain enough information to compute its partial derivatives of all orders at $x$, and thus its power series expansion at $x$. This power series expansion is zero if and only if the function is zero (since $Z$ is connected and contains a rational point, $Z_L$ is also connected). 
\end{proof}

Applying \cref{prop.rigid-zar-dense}, we obtain the following explicit example illustrating the failure of the n\"aive formulation of a $p$-adic Andr\'{e}-Oort conjecture for local Shimura varieties. 

\begin{example}\label{example.andre-oort-counterexample}
Let $Z \subseteq \mbb{P}^{n-1}_{\mbb{Q}_{p^n}}$ be a smooth closed connected curve of positive genus containing a $\mbb{Q}_{p^n}$-rational point --- for example, when $n=3$, $Z$ could be the closure in $\mbb{P}^2$ of the affine plane curve $y^2=x^3-1$.  Then, by \cref{prop.rigid-zar-dense}, for any $L/\Qpbreve$, special points are dense in $Z_L$ when we view $\mbb{P}^{n-1}_{L}$ as $\mbb{P}^{n-1,b-\adm}_{L}$. However, $Z_L$ is not a special subvariety: indeed, in this case, all special subvarieties are flag varieties for reductive groups (because $G_b$ is anisotropic mod center, but any $H \leq \GL_n$ giving rise to a special subvariety must also have an inner form realized as a subgroup of $G_b$, thus cannot have a non-trivial unipotent radical --- see also \cite[Remark 4.5.2]{howe-klevdal:ap-pahsI}), thus admits no non-zero global one-forms, whereas the positive genus curve $Z$ does admit a non-zero global one-form. From this we deduce also that, for any $K \leq \GL_n(\mbb{Q}_p)$ compact open, any connected component of the pre-image of $Z_L$ in the local Shimura variety $\mc{M}_{b,[\mu],K} \times_{\Spa\+ \Qpbreve} \Spa\+ L$ contains a Zariski dense set of special points but is not a special subvariety.  
\end{example}

\section{A bi-analytic Ax--Schanuel conjecture}\label{s.ax-schanuel}

The main purpose of this section is to formulate a bi-analytic Ax--Schanuel conjecture for infinite level basic local Shimura varieties, \cref{conj.ax-schanuel} (see also \cref{remark.variants} for some further generalizations). In the setting of complex bi-algebraic geometry \cite{klingler-ullmo-yafaev:bialgebraic}, Ax--Schanuel type theorems typically refine Ax--Lindemann type theorems, and in \cref{remark.ax-lindemann-relation} we explain how \cref{conj.ax-schanuel} can be viewed as a partial generalization of the bi-analytic Ax--Lindemann theorem of \cite{howe-klevdal:ap-pahsII}. 

We begin in \cref{ss.ax-schan-setup} by setting up some notation  --- in particular, as in \cite{Howe.CohomologicalSmoothnessConjecture, ivanov-weinstein}, we will work with a variant of infinite level local Shimura varieties obtained by fixing a determinant to kill off a locally profinite set of connected components. In \cref{ss.tangent-transverse} we then define transverse and connected transverse intersections using the theory of Banach--Colmez Tangent Bundles arising from the inscribed structures on these spaces as described in \cite{Howe.InscriptionTwistorsAndPAdicPeriods, Howe.CohomologicalSmoothnessConjecture}. To formulate \cref{conj.ax-schanuel}, the key new input is a notion of exceptional intersections. We use our study of transversality to make such a definition in \cref{ss.exceptional-intersections}, and then show that special subvarieties give rise to natural exceptional intersections (\cref{example.strictly-special-positive-dimension-exceptional} and \cref{example.special-points-exceptional}). In \cref{ss.ax-schanuel} we state the Ax--Schanuel conjecture, \cref{conj.ax-schanuel}, which says that all exceptional intersections are contained in special subvarieties, and relate it to the bi-analytic Ax--Lindemann theorem.

\subsection{Setup}\label{ss.ax-schan-setup}
Let $G/\mathbb{Q}_p$ be a reductive group, let $[\mu]$ be a minuscule conjugacy class of cocharacters of $G_{\overline{\mbb{Q}}_p}$, and let $b$ represent the unique basic class in $B(G,[\mu^{-1}])$. We fix a $p$-adic field $L/\Qpbreve([\mu])$ and let $C:=\overline{L}^\wedge$. Our Ax--Schanuel conjecture will not be in terms of the associated moduli space $\M_{b,[\mu], L}$, but instead in terms of the variant obtained by fixing a determinant as in \cite{Howe.CohomologicalSmoothnessConjecture} (following \cite{ivanov-weinstein} in the EL case). We recall this variant now: let $G^\der$ be the derived subgroup of $G$ and $G^\ab=G/G^\der$. We write $\det$ for the projection $G \rightarrow G^\ab$, which induces by functoriality a map $\det_{b,[\mu]}: \M_{b,[\mu]} \rightarrow \M_{\det(b), \det(\mu), \Qpbreve([\mu])}$. We fix a point $\tau: \Spd\+ C \rightarrow \M_{\det(b), \det(\mu), \Qpbreve([\mu])}$ that lies in the image of $\det_{b,[\mu]}$ (such a point $\tau$ exists as a consequence of, e.g., \cite[Proposition 7.10]{howe-klevdal:ap-pahsII}), and we write $\M^{\tau}:=\tau^* \M_{b,[\mu]}$, a diamond over $\Spd\+ C$. We write a superscript $\tau$ on the period maps $\pi, \pi_\HT, \textrm{ and } \pi_\Hdg$ to denote the period maps obtained by base change to $C$ and restriction to $\M^{\tau}$; in particular, we have 
\[ \pi^\tau= \pi_{\Hdg}^\tau \times \pi_{\HT}^\tau : \mc{M}^\tau \rightarrow  \Fl_{[\mu^{-1}],C}^{b-\adm,\diamond} \times_{\Spd\+ C} X_{[\mu],C}^\diamond \subseteq \Fl_{[\mu^{-1}],C}^\diamond \times_{\Spd\+ C} \Fl_{[\mu],C}^\diamond. \]
We write $\impi:=\Fl_{[\mu^{-1}]}\times_{\mathbb{Q}_p([\mu])} \Fl_{[\mu]}$, so that the target of $\pi^\tau$ can be taken as $\impi_C^\diamond$. A locally closed rigid analytic subvariety $S \subseteq \impi_{C}$ is called $\overline{L}$-analytic if it can be defined over $L'$ for $L'/L$ an algebraic extension, i.e. if it is the base change to $C$ of a rigid analytic subvariety of $\impi_{L'}$.

\subsection{Tangent bundles and transverse intersections}\label{ss.tangent-transverse} We describe some constructions following \cite{Howe.CohomologicalSmoothnessConjecture}: For $\mf{g}^\circ := \Lie\+ G^\der$, we write $\mc{E}^\circ_{\mr{max}}$ for the vector bundle on the relative Fargues--Fontaine curve over $\M^\tau$ that is the minimal common modification of $\mf{g}^\circ \otimes_{\mbb{Q}_p} \mc{O}_\FF$ and $\mc{E}(\mf{g}_b^\circ)$, where $\mf{g}^\circ_b$ denotes the isocrystal $(\mf{g}^\circ_{\Qpbreve}, \mr{Ad}(b)\sigma)$ (note that, since $b$ is basic, this isocrystal is semistable of slope zero so $\mc{E}(\mf{g}^\circ_b)=(\mf{g}^\circ_b)^{\mr{Ad}(b)\sigma=1} \otimes_{\mathbb{Q}_p} \mathcal{O}_\FF$). 
We define the Banach Colmez Tangent Bundle of $\M^\tau$ as 
\[T_{\M^\tau} := \BC(\mc{E}^\circ_{\mr{max}}).\]
We will then define the derivative of $\pi^\tau$ as a map
\[ d\pi^{\tau}: T_{\M^\tau}=\BC(\mc{E}_{\max}^\circ) \rightarrow {\pi^{\tau}}^* T_{\impi} \]
 where $T_D$ is viewed as the vector bundle over $D^\diamond$ obtained from the rigid analytic tangent bundle of $D$. We first observe that, from the definitions of the period maps, 
\[ {\pi^{\tau}}^* T_D=\gr^{-1}_\Hdg (\mf{g}_b^\circ  \otimes_{\Qpbreve}\mathcal{O}) \oplus \gr^{-1}_\HT (\mf{g}^\circ \otimes_{\mathbb{Q}_p} \mathcal{O}).  \]
Note that it is important that we are in the minuscule case so that the tangent bundles of the corresponding flag varieties can be described as the $-1$ graded piece of the universal filtration on $\mf{g}$ or $\mf{g}^\circ$ (i.e., that this graded piece is equal to $\mf{g}^\circ/\Fil^0 \mf{g}^\circ$). 

Writing $\mathcal{E}^\circ_\mr{min}$ for the minimal common modification of  $\mf{g}^\circ \otimes_{\mbb{Q}_p} \mc{O}_\FF$ and $\mc{E}(\mf{g}_b^\circ)$ and writing $\infty$ for the canonical Cartier divisor on the relative Fargues--Fontaine curve, we claim there is a natural modification exact sequence
\begin{equation}\label{eq.dpitau-mod-seq} 0 \to \mc{E}_{\min}^\circ \to \mc{E}^\circ_{\mr{max}} \to \infty_\ast \pi^{\tau \ast} T_D \to 0. \end{equation}
We then define $d\pi^\tau$ to be the map 
\[ \BC(\mc{E}^\circ_{\mr{max}}) \rightarrow  {\pi^{\tau}}^* T_D=\gr^{-1}_\Hdg (\mf{g}_b^\circ  \otimes_{\Qpbreve}\mathcal{O}) \oplus \gr^{-1}_\HT (\mf{g}^\circ \otimes_{\mathbb{Q}_p} \mathcal{O}) \]
obtained by taking sections over the relative Fargues--Fontaine curve in \cref{eq.dpitau-mod-seq}. 

To obtain the modification sequence in \cref{eq.dpitau-mod-seq}, we can compute in coordinates. Let $P= \Spa(A, A^+)$ be an affinoid perfectoid space. A map $P \to \mc{M}^\tau$ gives a modification $\mc{E}_b \dashrightarrow \mc{E}_\triv$ bounded by $[\mu]$ which, upon restricting to $\mbb{B}_\dR(A)$ and using the canonical trivializations of $\mc{E}_b, \mc{E}_\triv$ at $\mbb{B}_\dR$, yields an element $\alpha \in G(\mbb{B}_\dR(A))$. For simplicity, write $(\mbb{B}_\dR(A), \mbb{B}_\dR^+(A)) = (B, B^+)$ and let $t \in B^+$ be a generator of $\ker(\theta \colon B^+ \to A)$. As in \cite[Proposition 6.2.3]{Howe.InscriptionTwistorsAndPAdicPeriods}, we form the lattice $\mc{L}_{\max}(P) = \alpha \mf{g}^\circ_{B^+} + \mf{g}^\circ_{B^+}$ and then $\mc{E}^\circ_{\max}|_P$ is the modification of $\mc{E}_\triv(\mf{g}^\circ)|_P$ at $\mc{L}_{\max}(P)$ (which is the same as the modification of $\mc{E}_b(\mf{g}^\circ)|_P$ at $\alpha^{-1}\mc{L}_{\max}(P)$). Similarly, $\mathcal{E}^\circ_{\min}$ is the modification at $\mc{L}_{\mr{min}}(P) = \alpha \mf{g}^\circ_{B^+} \cap \mf{g}^\circ_{B^+}$.

We have exact modification sequences
\begin{align*} 
0 \rightarrow \mc{E}_\triv(\mf{g}^\circ) \rightarrow \mc{E}_{\mr{max}}^{\circ} \rightarrow \infty_* \gr^{-1}_\Hdg (\mf{g}_b^\circ \otimes_{\mathbb{Q}_p} \mathcal{O}) \rightarrow 0 \\
0 \rightarrow \mc{E}(\mf{g}_b^\circ) \rightarrow \mc{E}_{\mr{max}}^{\circ} \rightarrow \infty_* \gr^{-1}_\HT (\mf{g}^\circ \otimes_{\Qpbreve} \mathcal{O}) \rightarrow 0. 
\end{align*}
Indeed, the first exact sequence amounts, after pulling back via a map $P \to \mc{M}^\tau$ (with notation as above), to the isomorphisms
    \[ \frac{\alpha \mf{g}^\circ_{B^+} + \mf{g}^\circ_{B^+}}{\mf{g}^\circ_{B^+}} \cong \left( \frac{\alpha \mf{g}^\circ_{B^+}}{t\alpha\mf{g}^\circ_{B^+}} \middle/  \frac{\alpha\mf{g}^\circ_{B^+} + \mf{g}^\circ_{B^+}}{t\alpha \mf{g}^\circ_{B^+}} \right) \cong \gr^{-1}_{\Hdg}(\mf{g}^\circ_A). \]
Here we use that $[\mu]$ is minuscule, which implies that $\mf{g}^\circ_{B^+}$ is contained in $t^{-1} \alpha \mf{g}^\circ_{B^+}$ and contains $t\alpha \mf{g}^\circ_{B^+}$. The second containment allows for the identification of the Hodge filtration in the second isomorphism, and from the first containment we have $\alpha \mf{g}^\circ_{B^+} \supseteq \alpha \mf{g}^\circ_{B^+} \cap \mf{g}^\circ_{B^+} \supseteq t\alpha \mf{g}^\circ_{B^+}$ which gives the first isomorphism. The second exact sequence above is arrived at in the same way. Combining these, we obtain an exact sequence of sheaves on the relative Fargues--Fontaine curve over $\mc{M}^\tau$
    \[ 0 \to \mc{E}_{\min}^\circ \to \mc{E}^\circ_{\mr{max}} \to \infty_\ast \pi^{\tau \ast} T_D \to 0, \] 
where the surjectivity of the last map reduces to a simple calculation at geometric points $P = \Spa(C) \to \mc{M}^\tau$ by --- using notation as above --- choosing a basis $e_1, \ldots, e_n$ of $\mf{g}_{B^+}^\circ$ so that $\alpha \mf{g}^\circ_{B^+}$ has basis $t^{-1}e_1, \ldots, t^{-1}e_{i-1}, e_i, \ldots, e_j, te_{j+1}, \ldots, te_{n}$.

\begin{remark} \label{remark.inscribed-origin-def}
In \cite{Howe.CohomologicalSmoothnessConjecture} it is explained that $T_{\M^\tau} $
is the Banach Colmez Tangent Bundle associated to a differentiable structure on $\M^\tau$, an \emph{inscription} in the terminology of \cite{Howe.InscriptionTwistorsAndPAdicPeriods}, and the definition of  $d\pi^\tau$ we give here comes from computing the derivative of an inscribed upgrade of the period maps (cf. \cite[Corollary 9.2.3]{Howe.InscriptionTwistorsAndPAdicPeriods}). We have avoided the explicit use of the inscribed theory here in order to make our computations and conjectures accessible via more broadly familiar objects.  
\end{remark}

\subsubsection{}Given a smooth (locally closed) rigid analytic subvariety $\iota:S \hookrightarrow \impi_C$, we consider the intersection
\[ \mc{M}^\tau \cap S := \mc{M}^\tau \times_{\impi_C^\diamond} S^\diamond. \]
We define its Banach Colmez Tangent Bundle by
\[ T_{\mc{M}^\tau \cap S}:=\BC(\mc{E}_{\mr{max}}^\circ)|_{\mc{M}^\tau \cap S}
\times_{{\pi^\tau}^*T_{\impi}|_{\M^\tau \cap S}}{ {\pi^\tau}^*T_S |_{\M^\tau \cap S}}\]
(which can be arrived at as in \cref{remark.inscribed-origin-def} by taking an inscribed fibered product). 

We can rewrite this as another Banach--Colmez space of global sections: let $\mc{E}^\circ_{S}$ be the vector bundle on the relative Fargues--Fontaine curve over $\M^\tau \cap S$ obtained by modifying $\mc{E}_{\mr{max}}^\circ$ along $T_{S} \subseteq T_{\impi}$. Then, for $N_S=T_{\impi}/T_{S}$ the normal bundle of $S$ in $\impi$, $\mc{E}^\circ_S$ sits in the exact sequence  
\begin{equation}\label{eq.intersection-tangent-bundle-sequence} 0 \rightarrow \mc{E}^\circ_S \rightarrow \mc{E}_{\mr{max}}^\circ|_{\mc{M}^\tau \cap S}\rightarrow \infty_* {\pi^\tau}^* (N_{S}) \rightarrow 0, \end{equation}
and we have $T_{\mc{M}^\tau \cap S}=\BC(\mc{E}^\circ_S)$.

\begin{definition}\label{def.transverse}
For a rank one geometric point $x: \Spd\+ C' \rightarrow \mc{M}^\tau \cap S$, we say $\mc{M}^\tau$ intersects $S$ \emph{transversely at $x$} if the map $d\pi^\tau + d\iota: T_{\M^\tau, x} \oplus T_{S,\pi^\tau(x)} \rightarrow T_{\impi, \pi^\tau(x)}$ is surjective (it is equivalent to ask this map be surjective as a map of Banach--Colmez spaces or as a map of $\mathbb{Q}_p$-vector spaces). We say $\mc{M}$ intersects $S$ \emph{transversely} if $\mc{M}$ intersects $S$ transversely at every rank one geometric point. 
\end{definition}

\begin{lemma}
For a rank one geometric point $x: \Spd\+ C' \rightarrow \mc{M}^\tau \cap S$, $\mc{M}^\tau$ intersects ${S}$ transversely at $x$ if and only if $x^* \mc{E}^\circ_S$ has non-negative Harder--Narasimhan slopes. 
\end{lemma}
\begin{proof} Transversality of the intersection is equivalent to requiring that $d\pi^\tau$ induces a surjection $T_{\M^\tau, x} \rightarrow N_{S,\pi(x)}$. Applying the long exact sequence of cohomology on $\FF_{C'}$ to the pullback of the short exact sequence \cref{eq.intersection-tangent-bundle-sequence} along $x$ we find this is equivalent to the vanishing of $H^1(\FF_{C'}, x^*\mc{E}^\circ_S)$, which is equivalent to $x^*\mc{E}^\circ_S$ having non-negative Harder--Narasimhan slopes. 
\end{proof}

\begin{definition}\label{def.connected-transverse}
For a rank one geometric point $x: \Spd\+ C' \rightarrow \mc{M}^\tau \cap S$, we say $\M^\tau$ intersects $S$ \emph{connected-tranversely at $x$} if $x^*\mc{E}^\circ_{S}$ has positive Harder--Narasimhan slopes. We say $\M^\tau$ intersects $S$ \emph{connected-transversely} if $\M^\tau$ intersects $S$ connected-transversely at every rank one geometric point. 
\end{definition} 

\begin{remark}
    An intersection that is transverse at $x$ is connected-transverse at $x$ if and only if $T_{\M^\tau \cap S,x}$ is connected, explaining the terminology. 
\end{remark}

\begin{remark}
    By the semi-continuity of the Harder--Narasimhan polygon as in \cite[Theorem II.2.9]{fargues-scholze:geometrization}, the locus of rank one geometric points where the intersection is transverse (resp. connected-transverse) is the set of rank one points of a unique partially proper open in $\M^\tau$. 
\end{remark}

\begin{remark}\label{remark.slope-computation}
Fix a maximal torus in $G_{\overline{\mathbb{Q}}_p}$ and a Borel subgroup containing it, let $\mu$ to be the dominant representative for $[\mu]$, and let $2\rho$ the sum of the positive roots. At each geometric point $x:\Spd\+ C' \rightarrow\mc{M}^\tau \cap S$, $\mr{rank}\+x^*\mc{E}_{\mr{max}}^\circ=\dim G^\der$ and 
\[ \deg x^*\mc{E}_{\mr{max}}^\circ= \langle 2\rho,\mu \rangle = \dim X_{[\mu]}=\dim \Fl_{[\mu]}=\dim \Fl_{[\mu^{-1}]}=\dim \Fl_{[\mu^{-1}]}^{b-\adm}. \]
It follows that $\mr{rank}\+x^*\mc{E}_S^\circ=\dim G^\der$ and
\[ \mr{deg}\+ \mc{E}_S^\circ = \langle 2\rho,\mu \rangle -\mathrm{codim}_{\pi^\tau(x)} S_{C'} = \dim_{\pi^\tau(x)} S_C'-\langle 2\rho,\mu \rangle. \]
Thus the slope of $x^*\mc{E}_S^\circ$ is
\[ \frac{\langle 2\rho,\mu \rangle-\mathrm{codim}_{\pi^\tau(x)} S_{C'}}{\mr{dim}\+ G^\der}=\frac{\dim_{\pi^\tau(x)} S_{C'}-\langle 2\rho,\mu \rangle}{\mr{dim}\+ G^\der}. \]
In particular, when $S$ is equidimensional, this slope is constant equal to 
\[ \frac{\langle 2\rho,\mu \rangle-\mathrm{codim}S}{\mr{dim}\+ G^\der}=\frac{\dim S-\langle 2\rho,\mu \rangle}{\mr{dim}\+ G^\der}. \]
\end{remark}

\subsection{Exceptional intersections}\label{ss.exceptional-intersections}

We continue with the notation of the previous subsection.

\begin{definition}\label{def.exceptional}
    A non-empty connected component of $Z \subseteq \mc{M}^\tau \cap S$ is \emph{exceptional} if, for each rank one geometric point $z$ of $Z$ and $s:=\pi^\tau(z)$,
    \begin{enumerate}
        \item  if $\dim_s S >  \langle 2\rho, \mu \rangle$ then the intersection is not connected transverse at $z$  and
        \item if $\dim_s S \leq  \langle 2\rho, \mu \rangle$ then the intersection is not transverse at $z$ .
    \end{enumerate}
\end{definition} 

\begin{remark}\label{remark.dimension-exceptionality}
In the setting of \cref{def.exceptional}, if $\dim_{s} S < \langle 2\rho, \mu \rangle$ then \cref{remark.slope-computation} implies that the intersection cannot be transverse at $z$ (since a bundle of slope $< 0$ cannot have all non-negative Harder--Narasimhan slopes). Thus the second condition only needs to be checked at points such that $\dim_s S=\langle 2\rho, \mu \rangle$, where \cref{remark.slope-computation} implies the intersection is transverse if and only if $z^*\mc{E}_S^\circ$ is semistable of slope zero (since a vector bundle of slope $0$ whose Harder--Narasimhan slopes are all non-negative is semistable of slope $0$). 
Note that this depends only on $s$, and not the choice of $z$ lying above it in $\M^\tau$. 
\end{remark}

\begin{remark}
    When $\mathrm{dim}\, Z(G) \neq 0$, if we worked with $\M_C$ instead of $\M^\tau$ then none of these intersections would be connected transverse --- in that case, the central action of $Z(G)(\mathbb{Q}_p)$ induces slope zero directions in all intersections. This is our main reason for passing to $\M^\tau$ in the formulation of \cref{conj.ax-schanuel}. 
\end{remark}

In \cite[Definitions 6.14, 7.23, and 8.4]{howe-klevdal:ap-pahsII} we have defined notions of special subvarieties of $X_{[\mu],C}, \Gr_{[\mu],C}^{b-\adm}$ and $\mc{M}_C$ respectively which we recall here (note that the conditions on the weights appearing in these definitions in \cite{howe-klevdal:ap-pahsII} are automatic since we assume $[\mu]$ minuscule). 

\begin{definition}[Special subvarieties]\label{defn:special-subvarieties}
A special subdatum $(H, b_H, [\mu_H])$ consists of a closed subgroup $H \leq G$, $[\mu_H] \subseteq [\mu]$ a $H(\overline{\QQ}_p)$-conjugacy class of cocharacters of $H$, and $b_H \in H(\Qpbreve)$ representing the unique basic element of $B(H, [\mu_H^{-1}])$. A subset $Z$ of $X_{[\mu],C}$ (respectively $\Fl_{[\mu],C}^{b-\adm}$ or $\mc{M}_{C}$) is a \emph{special subvariety} if there exists special subdatum $(H, b_H, [\mu_H])$ such that $Z = X_{[\mu_H],C}$ (respectively $Z$ is $\Fl_{[\mu_H],C}^{b_H-\adm}$, or $Z$ is a connected component of $\mc{M}_{b_H, [\mu_H],C}$ --- in these cases, more precisely we mean that $Z$ is the image under a functoriality map associated to $g \in G(\breve{\QQ}_p)$ that $\sigma$-conjugates $b_H$ to $b$.). 

 We define a special subvariety of $\M^\tau_C$ to be a special subvariety of $\M_C$ that lies in $\M^\tau_C$. 
\end{definition}

\begin{definition}We say a special subvariety of $\mc{M}^\tau_C$ (resp. $\Fl_{[\mu^{-1}],C}^{b-\adm}$, resp. $X_{[\mu], C}$) is \emph{strictly special} if it is attached to a closed connected subgroup $H \leq G$ with $H^\der \neq G^\der$. 
\end{definition}

\begin{example}\label{example.strictly-special-positive-dimension-exceptional} Suppose $V$ is a strictly special subvariety of $\Fl_{[\mu^{-1}],L}^{b-\adm}$ of positive dimension corresponding to a subgroup $H \leq G$, and let $S=V \times_L X_{[\mu],L}$. Then, we claim that at any geometric point $\mc{E}_S^\circ$ admits a surjection onto $\left(\mf{g}^\circ/\mf{h} \cap \mf{g}^\circ\right) \otimes_{\mbb{Q}_p} \mc{O}_\FF$. Given this claim, since $V$ is strictly special, $\mf{g}^\circ \neq \mf{h} \cap \mf{g}^\circ$, and thus $\mc{E}^\circ_S$ cannot have all positive Harder--Narasimhan slopes at any point. In particular, we find $\M^\tau \cap S$ is exceptional if it is non-empty (note $\dim(S) > \dim(X_{[\mu]}) = \langle 2\rho, \mu\rangle$).

It remains to construct the surjection at geometric points
\[ \mc{E}_S^\circ \twoheadrightarrow \left(\mf{g}^\circ/\mf{h} \cap \mf{g}^\circ\right) \otimes_{\mbb{Q}_p} \mc{O}_\FF.\]
We explain how to construct this map in two equivalent ways. 

First, we give a conceptual explanation using the differential perspective via inscribed $v$-sheaves of \cref{remark.inscribed-origin-def}. If we write $V = \Fl_{[\mu_H^{-1}]}^{b_H-\adm}$ for the subdatum $(H, [\mu_H], b_H) \leq (G, [\mu], b)$ and write $\mc{M}^\tau_H:=\mc{M}_{b_H, [\mu_H], L} \times_{\mc{M}_b, [\mu], L} \mc{M}^\tau$, then 
\[ \M^\tau \cap S = \M_H^\tau \times^{H(\mathbb{Q}_p) \cap G^\der(\mathbb{Q}_p)} G^\der(\mathbb{Q}_p).\]
This description of $\mathcal{M}^\tau \cap S$ also applies at the level of inscribed $v$-sheaves, and the natural map from the push-out to $G^\der(\mathbb{Q}_p) / H(\mathbb{Q}_p)\cap G^\der(\mathbb{Q}_p)$ yields a surjection on Tangent Bundles after differentiating the map of inscribed $v$-sheaves. Because the associated vector bundles on the Fargues--Fontaine curve all have non-negative Harder--Narasimhan slopes, passing from these vector bundles on the Fargues--Fontaine curve to their Banach--Colmez spaces of global sections is fully faithful and it follows that at any geometric point $\mc{E}^\circ_S \twoheadrightarrow \left(\mf{g}^\circ / \mf{h} \cap \mf{g}^\circ\right) \otimes \mc{O}_\FF$ (see \cite[Corollary 3.2.4]{Howe.CohomologicalSmoothnessConjecture} for a similar argument).

For our second construction, we use the explicit description of $\mathcal{E}_S^\circ$. From the $H(\mathbb{Q}_p)$-equivariant maps $\mf{h}\cap \mf{g}^\circ \rightarrow \mf{g}^\circ \rightarrow \mf{g}/\mf{h}\cap \mf{g}^\circ $, we obtain a commutative diagram of modification exact sequences over $\mathcal{M}^\tau \cap S$:
\[\begin{tikzcd}
	0 & {\mf{g}^\circ \otimes_{\mathbb{Q}_p} \mathcal{O}_\FF} & { \mc{E}^\circ_S} & {\infty_* T_{V}=\gr^{-1}_\Hdg(\mf{h}\cap \mf{g}^\circ)} & 0 \\
	0 & {\mf{g}^\circ \otimes_{\mathbb{Q}_p} \mathcal{O}_\FF} & {\mc{E}_{\mr{max}}^\circ} & {T_{\Fl_{[\mu^{-1}]}}=\gr^{-1}_\Hdg\mf{g^\circ}} & 0 \\
	0 & {(\mf{g}^\circ/\mf{h}\cap\mf{g}^\circ) \otimes_{\mathbb{Q}_p} \mathcal{O}_\FF} & { \mc{E}'} & {\gr^{-1}_\Hdg(\mf{g}^\circ/\mf{h}\cap\mf{g}^\circ)} & 0
	\arrow[from=1-1, to=1-2]
	\arrow[from=1-2, to=1-3]
	\arrow[from=1-2, to=2-2]
	\arrow[from=1-3, to=1-4]
	\arrow[from=1-3, to=2-3]
	\arrow[from=1-4, to=1-5]
	\arrow[from=1-4, to=2-4]
	\arrow[from=2-1, to=2-2]
	\arrow[from=2-2, to=2-3]
	\arrow[from=2-2, to=3-2]
	\arrow[from=2-3, to=2-4]
	\arrow[from=2-3, to=3-3]
	\arrow[from=2-4, to=2-5]
	\arrow[from=2-4, to=3-4]
	\arrow[from=3-1, to=3-2]
	\arrow[from=3-2, to=3-3]
	\arrow[from=3-3, to=3-4]
	\arrow[from=3-4, to=3-5]
\end{tikzcd}\]
The composition of the rightmost column of arrows is zero, so the composition of the middle column factors through $(\mf{g}^\circ/\mf{h}\cap\mf{g}^\circ) \otimes \mathcal{O}_\FF$, giving the desired map. 
\end{example}

\begin{example}
On the other hand, in the notation of \cref{example.strictly-special-positive-dimension-exceptional}, if $V$ is a zero dimensional special subvariety then it is a single special point and $\mc{M}^\tau \cap S$ is the fiber of $\pi_\Hdg$ over this point. Thus it is identified with $G^\der(\mathbb{Q}_p)$, whose Tangent Bundle is $\ul{\mf{g}^\circ}$, and $\mc{E}_S^\circ=\mf{g}^\circ \otimes_{\mbb{Q}_p} \mathcal{O}_{\FF}$. It follows that in this case  $\mc{M}^\tau \cap S$ is \emph{not} exceptional --- this is good because the same argument applies if $V$ is any point of $\Fl_{[\mu^{-1}]}^{b-\adm}$, not just a special point! 
\end{example}

The following shows special points still arise as exceptional intersections. 

\begin{example}\label{example.special-points-exceptional}
Suppose $S=\Spa\+ C$ is an $\overline{L}$-analytic-point, i.e. $S$ is  a $C$-point coming from a classical point of $\Fl_{[\mu^{-1}],L}^{b-\adm} \times_L X_{[\mu],L}$. Then, by \cite[Theorem B]{howe-klevdal:ap-pahsI}, $\mc{M}^\tau \cap S$ is non-empty if and only if it is a locally profinite set of special points. On the other hand, if $[\mu]$ is not central then $\langle 2\rho, \mu \rangle >0$ and thus, by \cref{remark.dimension-exceptionality}, the intersection is also exceptional if and only if it is non-empty.  
\end{example}

\subsection{The Ax--Schanuel conjecture}\label{ss.ax-schanuel}
\cref{example.strictly-special-positive-dimension-exceptional} and \cref{example.special-points-exceptional} showed that strictly special subvarieties give rise to natural exceptional intersections. The following Ax--Schanuel conjecture says that, conversely, strictly special subvarieties should account for all exceptional intersections of $\M^\tau$ with smooth $\overline{L}$-analytic subvarieties of $\impi$. 

\begin{conjecture}[Minuscule bi-analytic Ax--Schanuel]\label{conj.ax-schanuel}Let $G/\mathbb{Q}_p$ be a reductive group, let $[\mu]$ be a minuscule conjugacy class of cocharacters of $G_{\overline{\mbb{Q}}_p}$, and let $b \in G(\Qpbreve)$ represent the unique basic class in $B(G,[\mu^{-1}])$. Let $L/\Qpbreve([\mu])$ be a $p$-adic field, and for $\impi$ as defined in \cref{ss.ax-schan-setup}, suppose $S \subseteq \impi_C$ is a smooth locally closed $\overline{L}$-analytic subvariety. If $Z \subseteq \mc{M}^\tau \cap S$ is an exceptional component, then $Z$ is contained in a strictly special subvariety of $\mc{M}^\tau_C$.
\end{conjecture}

\begin{remark}
For $Z \subseteq \mc{M}_C$ connected, the following are equivalent:
\begin{enumerate}
    \item $Z$ is contained in a strictly special subvariety of $\mc{M}_C$
    \item $\pi_{\Hdg}(Z)$ is contained in a strictly special subvariety of $\Fl_{[\mu^{-1}],C}^{b-\adm}$
    \item $\pi_\HT(Z)$ is contained in a strictly special subvariety of $X_{[\mu],C}$
\end{enumerate}
Indeed, this follows as the special subvarieties in $\mc{M}_C$ are precisely the connected components of the pre-images under $\pi_\Hdg$ (resp. $\pi_\HT$) of special subvarieties in $\Fl_{[\mu^{-1}],C}^{b-\adm}$ (resp. $X_{[\mu],C}$). Thus, one can reformulate the prediction of \cref{conj.ax-schanuel} using any of these equivalent conditions. 
\end{remark}

\begin{remark}
If the motivic Galois group of the universal $G$-admissible pair over $\Fl_{[\mu^{-1}]}^{b-\adm}$ does not have derived subgroup equal to $G^\der$, then \cref{conj.ax-schanuel} is vacuous since in this case any connected component of $\mathcal{M}^\tau$ will itself be a strictly special subvariety according to our definition. This can occur: for example, one can consider any data $(G,b,[\mu])$ and then take its product
with the trivial data for $\mathrm{SL}_2$. However, one can always pass to the moduli space where $G$ is replaced with the motivic Galois group instead in order to obtain a more interesting conjecture.  
\end{remark}

\begin{remark}[Relation with bi-analytic Ax--Lindemann]\label{remark.ax-lindemann-relation}
 Let $V$ be a smooth subvariety of $\Fl_{[\mu^{-1}],C}^{b-\adm}$ of dimension $>0$ and let $S=V \times_C X_{[\mu],C}$.  Suppose $Z \subseteq \mc{M}^\tau \cap S = \pi_{\Hdg}^{-1}(V)$ is a connected component that is not exceptional. Let $z \in Z$ be a point where the intersection is connected transverse. Since $T_{\mc{M}^\tau \cap S,z}$ has no profinite directions, we expect\footnote{We emphasize that we do not know this discreteness of connected components; this kind of statement is essentially a generalization of the cohomological smoothness conjecture of \cite{Howe.CohomologicalSmoothnessConjecture}.} the set of connected components of $\mc{M}^\tau \cap S$ to be discrete at $z$ and thus for the natural inscribed structure on $Z$ to have the same tangent bundle, i.e. that $T_{Z,z} = T_{\mc{M}^\tau \cap S,z}$. Since this bundle has non-negative slopes, its global sections would span the specialization at $\infty$. From this it would follow that $d\pi_{\HT}(T_{Z,z})$ spans $T_{X_{[\mu]},\pi_{\HT}(z)}.$  In particular, this heuristic implies that $\pi_{\HT}|_{Z}$ should not factor through any rigid Zariski closed subvariety of $X_{[\mu]}$.

If this heuristic can be made precise then, together with \cref{conj.ax-schanuel}, it would imply that for any positive dimensional smooth subvariety $V \subseteq \Fl_{[\mu^{-1}],L}^{b-\adm}$ and connected component of $\pi_{\Hdg}^{-1}(V_C)$, the $\overline{L}$-analytic rigid Zariski closure in $X_{[\mu],C}$ of $\pi_{\HT}(Z)$ is a special subvariety. Indeed, applying \cref{conj.ax-schanuel} we find that we can pass everywhere to a special subvariety where the intersection is no longer exceptional, and then apply the above to obtain Zariski density of the image of $\pi_\HT$. Note that this statement is very closely related to \cite[Theorem A]{howe-klevdal:ap-pahsII} (the same heuristic argument will also apply in the setting of \cref{remark.variants}-(1), giving a statement that is even closer). 
\end{remark}

\begin{remark}[Variants]\label{remark.variants}
We now discuss extensions of \cref{conj.ax-schanuel}; we have focused on the simpler case above rather than these extensions in order to emphasize the analogy with Ax--Schanuel theorems in complex bi-algebraic geometry. 
\begin{enumerate}
    \item For $K \leq G(\mathbb{Q}_p)$ compact open and $K_b \leq G_b(\mathbb{Q}_p)$ compact open, we could replace $\Fl_{[\mu^{-1}]}^{b-\adm}$ with the local Shimura variety $\mathcal{M}_{[\mu], K}$ whose associated diamond is $\M/K$ and $X_{[\mu]}$ with the local Shimura variety $\M_{[\mu^{-1}], K_b}$ whose associated diamond is $\M/K_b$ (this is obtained by the basic duality of \cite[\S8.5]{howe-klevdal:ap-pahsII} as a local Shimura variety for the group $G_b$, and $[\mu^{-1}]$ is viewed as a conjugacy class of cocharacters of the latter). Replacing $\pi_{\Hdg}$ with the projection map $\pi_K : \M \rightarrow \M/K$ and $\pi_{\HT}$ with the projection map $\pi_{K_b}: \M \rightarrow \M/K_b$, one can carry out all of the above constructions for the induced map $\M^\tau \rightarrow (\mathcal{M}_{[\mu], K, C} \times_C \mathcal{M}_{[\mu], K_b, C})^\diamond$. In particular, we expect \cref{conj.ax-schanuel} to hold also for exceptional intersections of $\M^\tau$ with $\overline{L}$-analytic smooth subvarieties $S \subseteq \mathcal{M}_{[\mu], K, C} \times_C \mathcal{M}_{[\mu^{-1}], K_b, C}.$ 
   \item Fixing a representative $\mu$ for $[\mu]$ and writing its centralizer as $M$, there are canonical $M$-torsors over $\Fl_{[\mu^{-1}],C}$ and $\Fl_{[\mu], C}$ whose pullbacks to $\M_C$ are canonically isomorphic after fixing a trivialization of $\mathbb{Z}_p(1)$ over $C$ (this is an incarnation of the the Hodge-Tate comparison). There is a universal rigid analytic variety $\impi'/\impi_C$ for this property, and thus a natural map $q: \M_C \rightarrow {\impi'}^\diamond$. After promotion to a map of inscribed $v$-sheaves, it induces an isomorphism 
   \[ dq \boxtimes \mc{O}: T_{M_C} \boxtimes \mathcal{O} \xrightarrow{\sim}  q^*T_{\impi'}. \]
   This is discussed in more detail and greater generality in \cite{howe.p-adic-manifold-fibrations}. Because of the twist $\impi'$ does not admit a model over a finite extension of $L$, but it does admit natural ineffective descent data so that it still makes sense to discuss $\overline{L}$-analytic subvarieties as those rigid analytic subvarieties of $\impi'$ that are invariant under an open subgroup of $\Gal(\overline{L}/L)$; these include, in particular, the preimages-of $\overline{L}$-analytic subvarieties of $\impi_C$. We can define tangent bundles, intersections with $\M^\tau$, (connected) transversality, exceptional intersections, etc. more generally for these $\overline{L}$-rigid analytic subvarieties of $\impi'$ in order to also formulate \cref{conj.ax-schanuel} in this setting. Even more generally, one could also pullback $\impi'$ to the product of local Shimura varieties as in part (1) to obtain a conjecture in this setting. 
\item When $[\mu]$ is not minuscule, the natural period domains $X_{[\mu]}$ and $\Gr_{[\mu]}^{b-\adm}$ are opens not in flag varieties but instead in Schubert cells in the $\mathbb{B}^+_\dR$-affine Grassmannians. These latter are not rigid analytic, but as in \cite{howe-klevdal:ap-pahsII}, it still makes sense to discuss $\overline{L}$-analytic rigid analytic sub-diamonds (which can be moreover characterized in terms of maps to the flag-variety satisfying Griffiths transversality). In particular, one could also formulate \cref{conj.ax-schanuel} and the variants (1) and (2) above in this setting. 
\end{enumerate} 
\end{remark}

\bibliographystyle{plain}
\bibliography{refs}

\end{document}